\documentclass{amsart}

\usepackage{amssymb,amsfonts}
\usepackage{amsmath}
\usepackage[all,arc]{xy}
\usepackage{enumerate}
\usepackage{mathrsfs}
\usepackage{hyperref}
\usepackage{graphicx}
\usepackage{mathtools}
\usepackage{natbib}
\graphicspath{ {./Images/} }
\usepackage{pgfplots}
\pgfplotsset{compat=1.18}

\newtheorem{theorem}{Theorem}[section]
\newtheorem{corollary}[theorem]{Corollary}
\newtheorem{proposition}[theorem]{Proposition}
\newtheorem{lemma}[theorem]{Lemma}
\newtheorem{conjecture}[theorem]{Conjecture}

\theoremstyle{definition}
\newtheorem{definition}[theorem]{Definition}

\newtheorem{example}[theorem]{Example}

\theoremstyle{remark}
\newtheorem{remark}[theorem]{Remark}

\newcommand{\R}{\mathbf{R}}

\newcommand{\C}{\mathbf C}

\makeatletter
\let\c@equation\c@theorem
\makeatother
\numberwithin{equation}{section}

\usepackage{tikz-cd}

\newcommand{\cL}{\mathcal L}

\newcommand{\cB}{\mathcal B}

\newcommand{\cF}{\mathcal F}
\newcommand{\cG}{\mathcal G}

\newcommand{\cN}{\mathcal N}

\newcommand{\cU}{\mathcal U}
\newcommand{\cV}{\mathcal V}

\newcommand{\im}{\mathrm {Im}}

\newcommand{\dist}{\mathrm {dist }}

\newcommand{\gr}{\operatorname{graph}}
\newcommand{\cone}{\operatorname{Cone}}

\begin{document}
\author{Guoran Ye}
\title{Special Lagrangians with Cylindrical Tangent Cones}
\begin{abstract}
    We construct new examples of special Lagrangian submanifolds $Y\subset \C^{n+1},n\geq 3$ in a neighborhood of the origin, with an isolated singularity, but with cylindrical tangent cone $C\times\R$. Moreover, $Y\setminus\{0\}$ is connected while $(C\setminus\{0\})\times\R$ is not. Such examples exist, for example, when $C$ is a pair of transverse planes.
\end{abstract}

\maketitle

\section{Introduction}

A special Lagrangian submanifold is a Lagrangian submanifold in a Calabi-Yau manifold with constant Lagrangian angle. As calibrated submanifolds, special Lagrangians are not only minimal but also volume-minimizing in their homology class. Lying at the crossroads of symplectic geometry, complex geometry, and minimal surface theory, they have been central objects of study for decades. A key goal is to understand their moduli space. This naturally leads to the analysis of singularities, which arise on the boundary of this moduli space as limits of smooth special Lagrangians. Suppose $N$ has a singularity at the origin $0$. For any sequence $\rho_i\searrow 0$, the rescaled manifolds $\rho_i^{-1}N$ converge subsequentially to a special Lagrangian cone $C$. Such a cone $C$ is called a tangent cone of $N$ at $0$ and singularities are often classified according to their tangent cone(s).

The analysis of a tangent cone $C$ proceeds along two primary aspects. The first is existence: for instance, the existence of a special Lagrangian $L$ with tangent cone $C$ at a singularity, or the existence of one asymptotic to $C$ at infinity. The second is uniqueness: this includes the uniqueness of the tangent cone $C$ itself (i.e., independence of the rescaling sequence $\rho_i\searrow 0$), and the uniqueness of special Lagrangians asymptotic to a given cone $C$. Much of the existing literature focuses on cones with smooth links, including a series of works by Joyce and others \cite{joyce2002special}, \cite{joyce2003specialV}, \cite{imagi2015uniqueness}, \cite{marshal2002deformations}, \cite{haskins2004special}, and more.

Less understood are tangent cones with non-isolated singularities. Among these, the simplest examples are cylindrical cones of the form $C\times\R\subset \C^{n+1}$ where $C\subset\C^n$ is a cone with a smooth link. Recently, Collins and Li in \cite{collinsli2023uniqueness} proved the uniqueness of certain cylindrical tangent cones for special Lagrangians.

Cylindrical tangent cones are also studied in neighboring areas. In the context of Lagrangian mean curvature flow, \cite{li2024singularity} showed the uniqueness of certain cylindrical tangent flows. For minimal hypersurfaces, a series of works including \cite{simon1994uniqueness}, \cite{szekelyhidi2020uniqueness}, \cite{firester2025uniqueness} proved uniqueness of certain cylindrical tangent cones, while \cite{szekelyhidi2021minimal} showed the existence of nontrivial examples with cylindrical tangent cones. Similarly, in the study of Calabi-Yau metrics, various existence and uniqueness results for cylindrical models have been obtained \cite{conlon2017new}, \cite{li2019new}, \cite{szekelyhidi2019degenerations}, \cite{chiu2022nonuniqueness}, \cite{szekelyhidi2020uniquenessCY}.

\subsection{Main result}

The main result of this paper is the construction of new, nontrivial examples of special Lagrangian submanifolds possessing an isolated singularity with a cylindrical tangent cone. Prior to this work, the only known examples in this class were those constructed by Joyce \cite{joyce2002u} in $\C^3$ under a $U(1)$-invariance assumption. Joyce also constructed explicit examples whose tangent cone \textit{at infinity} is a pair of planes intersecting along a line \cite{joyce2000evolvequadratics}.

Throughout the paper, we let $C=\cone(\Sigma)\subset \C^n,n\geq 3$ be a special Lagrangian cone with smooth link $\Sigma$.

\begin{theorem}\label{thm: main}
    Suppose there exists a smooth, exact, connected special Lagrangian submanifold $L$ asymptotic to $C$ with rate $\mu< 0$ and $0\notin L$. Then there exists a special Lagrangian submanifold $Y\subset\C^{n+1}$ in a neighborhood of the origin such that $Y$ is smooth away from $0$ and has unique tangent cone $C\times\R$ at $0$. Moreover, $Y\setminus\{0\}$ is connected.
\end{theorem}

\begin{remark}
    The condition $0\notin L$ is not a fundamental restriction. It appears solely in proposition \ref{prop: weight is graphical} to guarantee a clean inclusion of weighted spaces. It is possible that the condition can be removed by a modification of the weighted space definitions. On the other hand, we are unaware of any smooth, exact, special Lagrangian submanifold $L$ asymptotic to a cone with rate $\mu< 0$ that contains $0$ but is not the cone itself. It is possible that $0\notin L$ is already implied by the other assumptions.
\end{remark}

\begin{remark}
    While the minimal hypersurfaces with cylindrical tangent cones obtained by Székelyhidi \cite{szekelyhidi2021minimal} required a separate proof of the area-minimizing property, the $Y$ constructed here is automatically area-minimizing by the calibrated nature of special Lagrangians.
\end{remark}

The hypothesis of Theorem \ref{thm: main} is satisfied in some well-known cases. For a pair of transverse planes, Lawlor \cite{lawlor1989angle} constructed such an asymptotic smoothing $L$ (Lawlor necks). Haskins \cite{haskins2004special} proved the existence of analogous smoothings for pairs of cones of the form $\operatorname{Cone}(\Sigma)\cup \mathrm{Cone}(e^{i\pi/n}\Sigma)$. These existence results lead to the following corollary.

\begin{corollary}
    Let $C$ be either:
    \begin{enumerate}
        \item A pair of transverse planes $P_1\cup P_2$ with characterizing angle $|\theta|=\pi$, or
        \item A pair of cones $\operatorname{Cone}(\Sigma)\cup \mathrm{Cone}(e^{i\pi/n}\Sigma)$.
    \end{enumerate}
    Then, there exists a special Lagrangian submanifold $Y\subset\C^{n+1}$ in a neighborhood of the origin with an isolated singularity at the origin, where the unique tangent cone is $C\times\R$. Moreover, $Y\setminus\{0\}$ is connected.
\end{corollary}

\subsubsection{Connectedness of the link}

It is important to note that this paper lies outside the scope of the uniqueness theorem for tangent cones proved by Collins and Li \cite{collinsli2023uniqueness}. Their result requires the link $\Sigma$ of the cone $C$ to be connected. A straightforward application of the maximum principle to the Liouville potential shows that this connectedness condition rules out the existence of a smooth, exact asymptotic smoothing with rate $\mu<0$. Therefore, in our construction, $\Sigma$ is necessarily disconnected.

A key geometric significance of our result is that while the link, hence the tangent cone $(C\setminus\{0\}) \times\R$ is disconnected, $Y\setminus\{0\}$ is connected. This is achieved precisely because the asymptotic smoothing $L$ is connected, effectively connecting the disconnected ends of $C = \cone(\Sigma)$. Note that this phenomenon distinguishes the special Lagrangian case from the minimal hypersurface case, where the link must be connected by Frankel's theorem \cite{frankel1961manifolds}.

In fact, we conjecture that if $C$ has connected link and is stable (i.e. the only linear and quadratic harmonic functions on $C$ correspond to translations and rotations, see \cite{joyce2003specialV}), then there do not exist isolated singularities with cylindrical tangent cone $C\times\R$:

\begin{conjecture}
    Suppose $Y$ is an exact special Lagrangian submanifold in $\C^{n+1}$ with tangent cone $C\times\R$ at $0$ where $C$ is stable and has a smooth connected link, then $0$ is not an isolated singularity.
\end{conjecture}

\subsubsection{Connectedness of the smoothing}

If we relax the assumption that $L$ is connected—for instance, if $L = \bigcup L_m$ is asymptotic to $C = \bigcup C_m$ where each $L_m$ is asymptotic to $C_m$—then we could apply the theorem to each pair $(L_m,C_m)$ to obtain individual $Y_m$. Their union, $Y=\bigcup Y_m$, would then have $C\times\R$ as its tangent cone, and $Y\setminus\{0\}$ will have the same number of connected components as $L$.

We may consider a more degenerate case when $C=\bigcup_m P_m$ is a union of planes. By taking $Y$ as the union of small special Lagrangian graphs over the planes $P_m\times\R$, one can obtain examples of special Lagrangian manifolds with cylindrical tangent cone $C\times\R$ and an isolated singularity at the origin (see Appendix \ref{app: disconn exm}). These examples are, however, trivial in the sense that  $Y\setminus\{0\}$ is disconnected. They are essentially a union of smooth perturbations of the planes $\bigcup_m P_m\times\R$ that preserve the tangent plane at the origin, without connecting the distinct components.

\subsection{Overview of the construction}

The proof of the main theorem follows a general analytic strategy common in geometric analysis: first, construct a sufficiently accurate approximate solution $X$ with an isolated singularity and the correct tangent cone $C\times\R$; second, perturb $X$ to an exact special Lagrangian solution $Y$ by inverting the linearization of the Lagrangian angle operator. We obtain the approximate solution $X$ through a gluing construction, seemingly similar to the minimal hypersurfaces construction \cite{szekelyhidi2021minimal} from an analytic perspective. However, the adaptation to the special Lagrangian setting introduces profound differences that constitute a major focus of this paper.

\subsubsection{First attempt}

Let $a>1$ and $u$ denote the coordinates on the $\R$ factor of $C\times\R$. Naively smoothing the cylindrical cone by replacing each $u$-slice with a scaled copy $|u|^a L$ of the smoothing $L$ yields a submanifold $X_0$ with an isolated singularity and tangent cone $C\times\R$ at the origin (see Example \ref{exm: X_0} and Figure \ref{fig_of_X0}). This can be a model of the approximate solution near the singular axis $\{0\}\times\R$.

On the other hand, the linearization of the Lagrangian angle operator is the Laplacian, so in the region away from the singular axis $\{0\}\times\R$, a special Lagrangian with tangent cone $C\times \R$ up to first order should be the graph of $d\varphi_a$ for a homogeneous harmonic function $\varphi_a$. We define this graph as $X_1$.

A natural first attempt is to define $X$ by interpolating between $X_0$ and $X_1$ using a cutoff function on the $1$-forms defining then as graphs over $C\times\R$. However, this fails for two fundamental reasons specific to the Lagrangian setting:
\begin{itemize}
    \item The naive smoothing $X_0$ is \textit{not} Lagrangian.
    \item The interpolation of two Lagrangian graphs may \textit{not} be Lagrangian.
\end{itemize}

\begin{figure}
\begin{center}
\begin{tikzpicture}
    \begin{axis}[
        width=12cm,
        height=8cm,
        view={50}{25},
        xlabel={},
        ylabel={$u$},
        zlabel={},
        clip=false,
        % Remove axis numbers and grid
        xtick=\empty,
        ytick=\empty,
        ztick=\empty,
        grid=none,
        axis lines=center,
        axis line style={thin}
    ]
    % Upper sheet wireframe
    \addplot3[
        restrict z to domain*={0:5},
        mesh,
        domain=-1:1,
        domain y=-.3:.3,
        samples=25,
        samples y=25,
        black,
        thin
    ] ({x}, {y}, {sqrt(x^2 + (abs(y))^(8/3))});
    \end{axis}
\end{tikzpicture}
\caption{$X_0$ from Example \ref{exm: X_0}. We correct it to a Lagrangian submanifold $X_2$ in section \ref{subsec: lag corr}}
\label{fig_of_X0}
\end{center}
\end{figure}
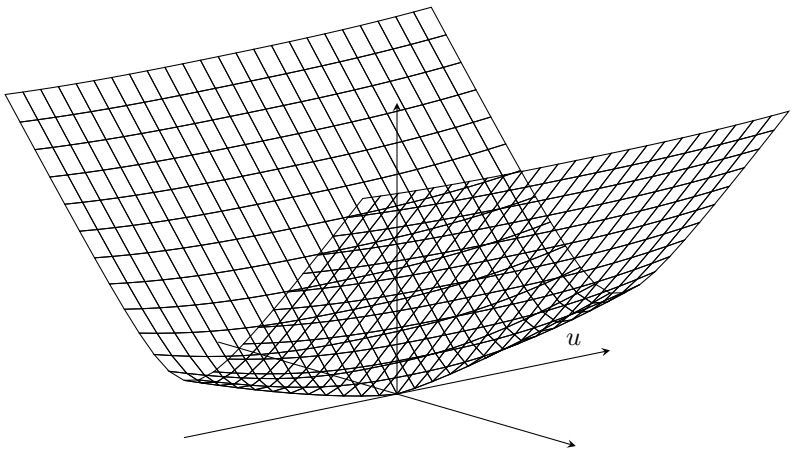

\subsubsection{Lagrangian correction and exactness}\label{subsubsec: Lag corr and exact}

There is a natural way of correcting $X_0$ to a Lagrangian $X_2$ that fixes the $u$-cross-sections (thus only changing the imaginary part of the last coordinate). A key observation is that such a correction exists precisely when $L$ is an \textit{exact} Lagrangian submanifold (see Remark \ref{rmk: prop of X_2}(3)).

\subsubsection{Cohomological information involved when gluing Lagrangian graphs}\label{subsubsec: glue}

In the minimal hypersurface case, one can often glue two graphical descriptions simply by using a cutoff function $\chi$ on the graphing functions themselves. In our setting, where we wish to glue two Lagrangian graphs, this naive approach fails. Suppose we have two Lagrangian submanifolds given as graphs of closed $1$-forms $\alpha_1$ and $\alpha_2$ over a Lagrangian $X$ (see Theorem \ref{thm: Lag nbhd} for the graph of closed $1$-forms). $\chi \alpha_1 + (1-\chi)\alpha_2$ in general is not closed, so its graph may not be Lagrangian. The correct approach, as detailed in the paper, is to work with \textit{potentials}. If the $1$-forms are exact, i.e., $df_1=\alpha_1$ and $df_2=\alpha_2$, then the graph of
\[
    \bar\alpha = d(\chi f_1 + (1-\chi)f_2) = \chi df_1 + (1-\chi)df_2 + d\chi (f_1-f_2)
\]
is Lagrangian. This reveals a critical requirement: the two graphical pieces must be exact in the gluing region. In our construction, this condition is guaranteed by $\mu < 0$ (see Remark \ref{rmk: exactness of graphicality}), or the exactness of $L$.

The presence of the $d\chi (f_1-f_2)$ term in $\bar\alpha$ introduces another layer of complexity. The potentials $f_1$ and $f_2$ are defined up to an additive constant on each connected component of the graph. While these constants do not affect $\alpha_1$ and $\alpha_2$, they directly impact the interpolating $1$-form $\bar\alpha$, thus affecting the geometry of its graph. A successful global construction, therefore, requires a canonical way to match these constants across the manifold. The relevant constants for our gluing are precisely the asymptotic limits of the Liouville potential $\beta_L$ from Lemma \ref{lem: beta_L near infty}, which are denoted $c_{\infty,m}$ for each connected component of $\Sigma_m$ of $\Sigma$.

\subsubsection{Perturbation}

The final approximate solution $X$, defined piecewise as the harmonic graph $X_1=\gr_{C\times\R}(d\varphi_a)$, the Lagrangian correction $X_2$, and the Lagrangian interpolation, has two key properties in a neighborhood of the origin.
\begin{enumerate}
    \item The Lagrangian angle $\theta_X$ decays rapidly near the singularity.
    \item The linearized operator $\Delta_X$ is invertible on suitable weighted function spaces.
\end{enumerate}
A standard contraction mapping argument (Proposition\ref{prop: main pert}) then yields a small perturbation $Y$ of $X$ that is genuinely special Lagrangian, completing the proof.

\subsection{Cohomological invariants and holomorphic discs}

This section is an attempt to explain the geometric intuition behind the issues that arise in the gluing step in relation to relevant notions in symplectic geometry. For a special Lagrangian smoothing $L$ of a cone $C = \cone(\Sigma)$, Joyce \cite{joyce2004specialI} (Definition 7.2) defined two cohomological invariants $Y(L)$ and $Z(L)$.

\begin{itemize}
    \item By \cite{joyce2004specialI} Theorem 7.6, $Y(L) = [\alpha]\in H^1(\Sigma\times [\Lambda,\infty))\cong H^1(\Sigma)$ where $\alpha$ is the $1$-form of which $L$ is the graph over $C$ near the asymptotic ends (see Remark \ref{rmk: exactness of graphicality}). Hence, in the discussion of \ref{subsubsec: glue}, $Y(L)=0$ means there is no obstruction to gluing $L$ at the asymptotic ends.
    \item Assuming $L$ is exact and connected, the Liouville potential $\beta_L$ is defined up to an additive constant. If we normalize it such that its asymptotic limits $c_{\infty,m}$ (see Lemma \ref{lem: beta_L near infty}) satisfy
    \begin{equation}\label{eqn: balancing cond}
        \sum_m \operatorname{vol}(\Sigma_m)c_{\infty,m} = 0.
    \end{equation}
    where the volume is computed as the $(n-1)$-dimensional volume of $\Sigma_m$ as a submanifold of the unit sphere $S^{2n-1}(1)\subset \C^n$, then following  \cite{joyce2004desingIII}, the invariant $Z(L)$ is precisely $(c_{\infty,m})\in \oplus_m\R \cong\oplus_m H^{n-1}(\Sigma_m) = H^{n-1}(\Sigma)$ under the normalization \eqref{eqn: balancing cond} (see (86) of \cite{joyce2004desingIII}).
\end{itemize}

We also note the relation between the cohomological invariants and holomorphic discs, which are important ingredients of the uniqueness of Lawlor necks \cite{imagi2015uniqueness}.

\begin{itemize}
    \item $Y(L)=0$ means any holomorphic disc bounded entirely by the asymptotically conical region of $L$ has zero area, i.e., such discs do not exist.
    \item The Lagrangian $L$ is exact means any holomorphic disc bounded entirely by $L$ has zero area, i.e., such discs do not exist. Moreover, $L$ is Hamiltonian isotopic to a scaling of itself (see Remark \ref{rmk: prop of X_2}(4)). This provides another interpretation of the existence of Lagrangian correction.
    \item $c_{\infty,m} - c_{\infty,m'}$ is the area of a holomorphic disc bounded by a triangle $L$, $\cone(\Sigma_m)$, and $\cone(\Sigma_{m'})$ (in case such a disc exists, see \cite{imagi2015uniqueness}). Fixing a scaling of $L$, under a scaling by $|u|^a$, the corresponding constants scale to $|u|^{2a} c_{\infty,m}$. This corresponds to the fact that the area of a holomorphic disc, which is real two dimensional, scales with the square of the scaling factor, i.e. $\operatorname{Area}(\lambda\Sigma) = \lambda^2\operatorname{Area}(\Sigma)$. This provides an interpretation of the dominant term $u^{2a}c_{\infty,m}$ of $G$ in Definition \ref{def: X_2 graph}, which cannot be obtained from the analytic perspective of viewing $X_2$ merely as a minimal smoothing of $C$ in each $u$-cross-section. When $L$ is a Lawlor neck asymptotic to $P_1\cup P_2$ with only two ends, \eqref{eqn: balancing cond} becomes $c_{\infty,1} = -c_{\infty,2} = A$. Up to a sign, $A$ is the sole parameter of the family of Lawlor necks (see \cite{imagi2015uniqueness}).
\end{itemize}

The paper is organized as follows. Section \S\ref{sec: Prelim} explains relevant preliminaries in special Lagrangian geometry. Section \S\ref{sec: Perturb} identifies sufficient conditions on an approximate solution $X$ that allow it to be perturbed to an actual solution. Finally, section \S\ref{sec: approx sol} explicitly constructs the approximate solution $X$ and verifies that it satisfies the sufficient conditions of section \S\ref{sec: Perturb}, thus proving the main theorem \ref{thm: main}. Section \S\ref{sec: approx sol} contains the main technical crux of the paper.

\subsection*{Acknowledgments} I would like to thank G\'{a}bor Sz\'{e}kelyhidi for suggesting this project and for his constant support over the years. This work was supported in part by NSF grant DMS-2506325.

\subsection{Notations}
Throughout the paper, we use the following notations:
\begin{itemize}
    \item We assume $n\geq 3$. Identify $\C^{n+1}$ with $\C^n\times \C$ equipped with standard coordinates $(z,w)\subset \C^n\times \C$, and write $z_j = x_j + iy_j$ for $j=1,\dots,n$ and $w=u+iv$. The standard symplectic forms are
    \[
        \omega_{\C^n} = \sum_{j=1}^n dx_j\wedge dy_j, \quad \omega_{\C^{n+1}} = \sum_{j=1}^n dx_j\wedge dy_j + du\wedge dv.
    \]
    \item Let $\Omega_n = dz_1\wedge\dots\wedge dz_n$ and $\Omega_{n+1} = dz_1\wedge\dots\wedge dz_n\wedge dw$ be the standard holomorphic volume forms. We say $N\subset \C^n$ is \textit{special Lagrangian} if $\im(\Omega_n)|_{N} = 0$; similarly, $X\subset \C^{n+1}$ is special Lagrangian if $\im(\Omega_{n+1})|_{X} = 0$. Equivalently, their Lagrangian angles $\theta_N$ and $\theta_X$ are identically $0$.
    \item We do not assume the link $\Sigma$ to be connected. We denote the connected components of $\Sigma$ by $\{\Sigma_m\}_{m=1}^M$, then $C = \bigcup_m\cone(\Sigma_m)$.
    \item By a \textit{cylindrical special Lagrangian cone} in $\C^{n+1}=\C^n\times\C$, we mean a product $C\times\R$ where $C=\operatorname{Cone}(\Sigma)\subset \C^n$ is a special Lagrangian cone in $\C^{n}$ with a smooth link $\Sigma$, and the $\R$ factor corresponds to the real line $\{v=0\}$ in the second factor of $\C^n\times\C$. We call it the singular axis $\R_u$.
    \item Denote by $\rho = |(z,w)|$ the radial function in $\C^{n+1}$ and by $r = |z|$ the radial function in $\C^n$. For a large constant $A$, we write $\rho^{-1}(0,A^{-1}) = \{\rho\in (0, A^{-1})\}$ for a small punctured neighborhood of the origin in $\C^{n+1}$.
    \item Fix a large integer $K$. We say a submanifold in $\C^n$ or $\C^{n+1}$ has \textit{$C^K$-bounded geometry} if the induced metric has bounded $C^K$-norm. Occasionally, we may omit $C^K$ and only say bounded geometry for the same meaning.
    \item Throughout the paper, $C,C_k,\dots$ denote positive constants that may change from line to line, but always depend only on stated quantities.
\end{itemize}

\section{Preliminary results}\label{sec: Prelim}

\subsection{Lagrangian neighborhood and McLean's theorem}

Given a Lagrangian submanifold $N$, we wish to describe the nearby Lagrangian submanifolds as graphs over $N$. We state the following version of the Lagrangian neighborhood theorem, adapted from Theorem 4.2 in \cite{joyce2004specialI}, which follows directly from the proof of Theorem 7.1 in \cite{weinstein1971symplectic} by Weinstein.

\begin{theorem}\label{thm: Lag nbhd}
    Let $N\subset\C^n$ be an embedded Lagrangian submanifold with bounded $C^K$-geometry. Suppose $N$ has a neighborhood $\cL_N$ with a fiber bundle structure $\cL_N = \{L_x: x\in N\}$, where each fiber $L_x$ is an embedded non-compact Lagrangian submanifold. Then there exists a fiber-preserving symplectomorphism $\Phi:U\rightarrow \cL_N$ from an open neighborhood $U$ of the zero section $N\subset T^*N$ onto its image in $\cL_N$.
    
    Moreover, if the fibers $L_x$ are perpendicular to $N$ at their intersections, i.e., $T_xL_x \perp T_xN$ for all $x\in N$, then
    \begin{equation}\label{eqn: normal Lag nbhd}
        \Phi(\alpha)[x] = x + J\alpha^{\sharp}(x) + Q\bigl(x,\alpha(x),\nabla\alpha(x)\bigr)
    \end{equation}
    for any $1$-form $\alpha$ on $N$, where $\alpha^{\sharp}$ is the dual vector field with respect to the ambient metric and $Q(x,y,z) = O(|y|^2+|z|^2)$ for small $(y,z)$.
\end{theorem}

We refer to $\Phi(U)$ as a \textbf{Lagrangian neighborhood} of $N$. Given a $1$-form $\alpha$ on $N$, we denote the image of $\alpha$ under $\Phi(\alpha)$ as the graph of $\alpha$:
\[
    \gr_N(\alpha) = \Phi(\alpha).
\]
Then $\gr_N(\alpha)$ is Lagrangian in $\C^n$ if and only if $\alpha$ is closed. In a perpendicular Lagrangian neighborhood as above, McLean's theorem \cite{mclean1998deformations} states that the Lagrangian angle of a nearby Lagrangian graph changes by $d^*\alpha$ up to at least quadratic terms:

\begin{theorem}[McLean]\label{thm: McL}
    In a Lagrangian neighborhood as above, the Lagrangian angle of $\gr_N(\alpha)$ has the form
    \[
        \theta_{\gr_N(\alpha)}\bigl(x,\alpha(x)\bigr) = \theta_N(x) + d^*\alpha(x) + Q\bigl(x,\alpha(x),\nabla \alpha(x)\bigr)
    \]
    where $Q(x,y,z) = O(|y|^2+|z|^2)$ for small $y,z$.
\end{theorem}

In this paper, we are only interested in Lagrangian deformations that are graphs of \textit{exact} forms $\alpha = df$ for some potential function $f:N \rightarrow\R$. For such deformations, we define the Lagrangian angle operator $LA_N$ at the level of potential functions.

\begin{proposition}\label{prop:Lag ang op}
    The Lagrangian angle operator has the form
    \[
        LA_N(f)[x] \equiv \theta_{\gr_N(df)}\bigl(x, df(x)\bigr) = \theta_N + \Delta_N f(x) + Q\bigl(x,df(x),\nabla df(x)\bigr),
    \]
    where $Q(x,y,z) = O(|y|^2+|z|^2)$ for small $y,z$.
\end{proposition}

To be more concrete, we use the following choices of neighborhood $\cL_N$ that are consistent with earlier literature.

\begin{enumerate}
    \item For a cone with a smooth link $C\subset\C^n$, we take the neighborhood $\cL_C$ to be the dilation-equivariant ones (see section 4.1 of \cite{joyce2004specialI} for details).
    \item For a cylindrical cone $C\times \R\subset\C^{n}\times\C$, we set $\cL_{C\times\R} = \cL_C\times\C$, which is equivariant under dilation in $\C^n$, as in (1).
    \item For other smooth Lagrangian submanifolds $N$, we choose $\cL_N$ to be the normal tubular neighborhood.
\end{enumerate}

\begin{remark}
    In cases (1) and (2), the Lagrangian neighborhoods obtained are equivalent to the adapted Darboux coordinates introduced in Section 2.3 of \cite{collinsli2023uniqueness}.
\end{remark}

\begin{remark}\label{rmk: Lag nbhd for r^-1 bdd}
    Note that the induced geometries on $C$ and $C\times\R$ are not $C^K$-bounded as required by the Lagrangian neighborhood theorem \ref{thm: Lag nbhd} and McLean's theorem. Instead, after scaling up by $r^{-1}$, they become bounded. We will later refer to this property as $r^{-1}$-bounded geometry (see Definition \ref{def: r^-1 bd geom}). Notice that the corresponding Lagrangian neighborhoods chosen above also become bounded after the same scaling. Consequently, for Lagrangians with $r^{-1}$-bounded geometry and for suitably weighted functions that remain bounded after rescaling by $r^{-1}$ (see Definition \ref{def: doubly weighted}), we can still take Lagrangian graphs (Proposition \ref{prop: 2,2 can take graph}), define the Lagrangian angle operator, and control the quadratic terms (Proposition \ref{prop: bdd op & Lip est}).
\end{remark}

\subsection{Exact special Lagrangian smoothing of $C$}

The key building block of the construction of this paper is a smooth, exact special Lagrangian submanifold asymptotic to a cone with rate $\mu< 0$.

\begin{definition}\label{def: AC_submfld}
    Let $C=\operatorname{Cone}(\Sigma)\subset \C^n$ be a special Lagrangian cone with smooth link $\Sigma$. A smooth special Lagrangian submanifold $L\subset \C^n$ is \textit{asymptotic to $C$ with rate $\mu$} if it is the graph of an exact $1$-form outside a large ball $B_\Lambda$, i.e., there exists $\Lambda_0 >0$ such that for all $\Lambda \geq \Lambda_0$,
    \[
        L\cap \{r\geq \Lambda\} = \gr_{C\cap \{r\geq \Lambda\}}(dg)
    \]
    for some $g:C\cap \{r\geq \Lambda\}\rightarrow \R$ such that for all $k\geq 0$,
    \begin{equation}\label{eqn: AC of g}
        |\nabla^kg| \leq C_k r^{\mu-k}.
    \end{equation}
    Here, $\nabla$ is computed using the induced (conical) metric $g_C = r^2 g_\Sigma + dr^2$.
\end{definition}

\begin{remark}\label{rmk: exactness of graphicality}
    A priori, an asymptotically conical Lagrangian $L$ is the graph of a closed, but not necessarily exact, $1$-form $\alpha$ over $C$ outside a large ball. Under the asymptotic rate considered ($|\alpha|=O(r^{\mu-1})$), the fact that any loop $\gamma\subset C\cap \{r\geq \Lambda\}$ is homotopic to the scaled up $\lambda\gamma\subset C\cap \{r\geq \Lambda\}$ implies for arbitrarily large $\lambda$,
    \[
        \int_\gamma \alpha = \int_{\lambda\gamma} \alpha \leq C\lambda^{\mu}.
    \]
    Letting $\lambda\rightarrow\infty$ , we obtain $\int_\gamma \alpha = 0$; hence $\alpha$ is exact, yielding a potential $g$. Apriori, $g$ is defined up to a constant on each connected component of $\Sigma_m$. The decay of $|dg|=|\alpha|=O(r^{\mu-1})$ implies the oscillation of $g$ tends to zero at infinity on each component, thus we may normalize $g$ by letting
    \[
        \lim_{x\in \cone(\Sigma_m),|x|\rightarrow\infty} g(x) = 0.
    \]
\end{remark}

Note that $(\C^n,\omega_{\C^n})$ is a Liouville manifold with the Liouville $1$-form $\lambda$ and the corresponding Liouville vector field $v_\lambda$ defined by
\[
    \lambda = \frac{1}{2}\left(\sum_{j=1}^n x_j\wedge dy_j-y_j\wedge dx_j\right), \quad v_\lambda = \frac{1}{2}\left(\sum_{j=1}^n x_j\frac{\partial}{\partial x_j} + y_j\frac{\partial}{\partial y_j}\right).
\]
In particular, if we denote the position vector as $\vec x$, then $v_\lambda = \frac{1}{2}\vec x$.

\begin{definition}
    A Lagrangian submanifold $L$ in $\C^n$ is \textit{exact} if the restriction of the Liouville $1$-form to $L$ is exact; i.e. there exists $\beta_L:L\rightarrow\R$ such that $d\beta_L = \lambda|_L$. We assumed $L$ is connected, so $\beta_L$ is defined up to a constant.
\end{definition}

Following Section 2.3 of \cite{collinsli2023uniqueness}, for an exact Lagrangian submanifold asymptotic to $C$ with rate $\mu<0$, we get the following relation between the Liouville $\beta_L$ and the potential $g$ of the Lagrangian graph over $C$. This result can also be viewed as a result of the co-area formula applying to a graph over a cone.

\begin{lemma}\label{lem: beta_L near infty}
    There exist constants $c_{\infty,m}$ for each $m=1,...M$ such that
    \[
        \beta_L(\Phi(dg)[x]) =  c_{\infty,m} - \frac{1}{2}r^3\frac{\partial}{\partial r}\left(\frac{g(x)}{r^2}\right)
    \]
    for all $x\in \cone(\Sigma_m)\cap \{r\geq \Lambda\}$. In particular, $\beta_L\rightarrow c_{\infty,m}$ when approaching an end of $L$ asymptotic to $\cone(\Sigma_m)$
\end{lemma}

\begin{corollary}\label{cor: asymp of beta_L}
    $\beta_L$ satisfies the global estimates
    \[
        |\beta_L| \leq C,\quad  |\nabla^k\beta_L| \leq C_k(1+r)^{\mu-k}\quad\text{for } k\geq 1
    \]
\end{corollary}

\begin{proof}
     In $\{r\geq \Lambda\}$, the estimate follows from Lemma \ref{lem: beta_L near infty} and \eqref{eqn: AC of g}. In $\{r\leq \Lambda\}$, which is compact, $\beta_L$ is uniformly bounded.
\end{proof}

\begin{example}[Lawlor necks]
    In $\C^n$, fix angles $\{\theta_k\in (0,\pi)\}$ with $\theta_1 + \dots + \theta_n = \pi$, and consider the two transverse planes
    \[
        P_1 = \{(z_1,\dots,z_n):z_1,\dots,z_n \in \R\},\quad 
        P_2 = \{(e^{i\theta_1}z_1,\dots,e^{i\theta_n}z_n):z_1,\dots,z_n \in \R\}.
    \]
    Then $C= P_1\cup P_2$ is a special Lagrangian cone. Lawlor \cite{lawlor1989angle} constructed a family of smooth, exact, special Lagrangians $L$ asymptotic to $C = P_1\cup P_2$ with rate $2-n$, widely referred to as \textit{Lawlor necks}.
\end{example}

\begin{example}[Haskins necks]
    Let $\Sigma_0$ be a special Legendrian submanifold in $S^{2n-1}(1)\subset \C^n$, and write $e^{i\pi/n}\Sigma_0 = \{(e^{i\pi/n}z_1,\dots,e^{i\pi/n}z_n):(z_1,\dots,z_n)\in \Sigma_0\}$. Then $C = \mathrm{Cone}(\Sigma_0) \cup \mathrm{Cone}(e^{i\pi/n}\Sigma_0)$ is a special Lagrangian cone. Haskins \cite{haskins2004special} constructed a family of smooth, exact, special Lagrangians $L$ asymptotic to $C = \mathrm{Cone}(\Sigma_0) \cup \mathrm{Cone}(e^{i\pi/n}\Sigma_0)$ with rate $2-n$.
\end{example}

From now on, we fix a particular exact special Lagrangian $L$ asymptotic to $C$ with rate $\mu< 0$. Having fixed $L$, we also fix the associated data from Definition \ref{def: AC_submfld}: the region $\{r\geq \Lambda_0\}$ where $L$ is graphical over $C$, and the potential function $g$. We also assume $0\notin L$, so up to a change of scaling, we may assume $\dist(0,L)\geq 1$.

\subsection{Homogeneous harmonic function on $C$}\label{subsec: harm fcn on C}

We recall some basic facts on homogeneous harmonic functions on the cone $C=\operatorname{Cone}(\Sigma)$. For $d \in \R$, a function $\upsilon:\operatorname{Cone}(\Sigma)\rightarrow \R$ is called \textbf{homogeneous of degree $d$} if $\upsilon(t\cdot) \equiv t^d \upsilon$ for all $t>0$. Equivalently, $\upsilon$ can be written as $\upsilon(\sigma,r)\equiv r^d \phi(\sigma)$ for some function $\phi:\Sigma\rightarrow \R$.

\begin{lemma}\label{lem:homogen harmonic}
    Let $C=\operatorname{Cone}(\Sigma)$ with $\Sigma$ smooth and $n \geq 3$. Suppose $\upsilon(\sigma,r)\equiv r^d \phi(\sigma)$ is a homogeneous function of degree $d$ on $C$ for $\phi:\Sigma\rightarrow \R$. Then
    \[
        \Delta_C \upsilon(\sigma,r) = r^{d - 2}(\Delta_\Sigma \phi - d(d + n -2)\phi).
    \]
    Hence, $\upsilon$ is harmonic on $C$ if and only if $\phi$ is an eigenfunction of $\Delta_\Sigma$ with eigenvalue $d(d + n -2)$.
\end{lemma}

\begin{definition}
    In the setting above, define
    \[
        D_\Sigma = \{d\in\R: d(d + n - 2)\text{ is an eigenvalue of }\Delta_\Sigma\}
    \]
    to be the set of degrees for which there exists a non-zero homogeneous harmonic function on $C=\operatorname{Cone}(\Sigma)$.
\end{definition}

Since $\Sigma$ is smooth and compact, the eigenvalues of $\Delta_\Sigma$ are non-negative. An immediate consequence is the following spectral gap.

\begin{lemma}\label{lem:spec gap}
    There are no non-constant homogeneous harmonic functions on $C=\operatorname{Cone}(\Sigma)$ with degrees in the interval $(2-n,0)$. i.e.,
    \[
        D_\Sigma \cap (2-n,0) = \varnothing.
    \]
\end{lemma}

\section{Perturbing in weighted spaces}\label{sec: Perturb}

Recall that to construct a special Lagrangian submanifold $Y$ with an isolated singularity modeled on $C\times \R$ in a neighborhood of the origin, we first build an approximate solution $X$ and then perturb it to an actual solution $Y$. This section establishes sufficient conditions on $X$ for the perturbation argument to succeed.

\subsection{\texorpdfstring{$r^{-1}$}{r^-1}-bounded geometry}

Observe that the geometry of $C\times \R$ degenerates when approaching the set $\{r=0\}$, but this degeneration is controlled in the sense that the induced metric becomes bounded when rescaled by $r^{-1}$. The following definition makes this notion precise.

\begin{definition}\label{def: r^-1 bd geom}
    In $\C^n$ and $\C^{n+1}$, we define \textit{$r^{-1}$-bounded geometry} as follows.
    \begin{enumerate}
        \item For a submanifold $L\subset \C^n$, set $\Omega_R(L) = \{r\in (R,2R)\}\cap L$. We say $L$ has \textit{$r^{-1}$-bounded $C^k$-geometry} if the $C^k$-norm of the rescaled metric $R^{-2}g_L=g_{R^{-1}L}$ is uniformly bounded on the sets $\Omega_R(L)$ for all $R>0$. Here $g_L$ is the induced metric on $L$, and $g_{R^{-1}L}$ is the induced metric on the rescaled submanifold $R^{-1}L$.
        \item For a submanifold $X\subset \C^{n+1}$, set $\Omega_{R,S}(X) = \{r\in (R,2R),\ \rho\in(S,2S)\}\cap X$. We say $X$ has \textit{$r^{-1}$-bounded $C^k$-geometry} if the $C^k$-norm of $R^{-2}g_X=g_{R^{-1}X}$ is uniformly bounded on $\Omega_{R,S}(X)$ for all $R,S>0$.
    \end{enumerate}
\end{definition}

The term \textit{$r^{-1}$-bounded geometry} reflects the fact that the induced geometry of $R^{-1}X$ is uniformly bounded. For brevity, we often write $\Omega_{R,S}=\Omega_{R,S}(X)$ when the submanifold $X$ is clear from context.

Besides trivial examples such as $C\times\R$, the following $X_0$ has $r^{-1}$-bounded geometry. $X_0$ serves as a baby model for part of the approximate solution $X$ that we construct in the next section (see Section \ref{subsec: gen pic} for more discussion).

\begin{example}\label{exm: X_0}
    Fix integer $a>1$ and let $L$ be a fixed smoothing asymptotic to $C$. Define $X_0 = \bigcup_{u\in\R} |u|^aL\times \{w=u\}\subset \C^{n+1}$ where $|u|^aL$ denotes the submanifold $L$ scaled by the factor $|u|^a$. In other words, for each $u\in\R$, the cross-section of $X_0$ in the slice $\{w = u\} \simeq \C^n$ equals $|u|^aL$ when $u\neq 0$, and equals the cone $C$ when $u=0$. Intuitively, $X_0$ has $r^{-1}$-bounded geometry since each of its $u$-cross-sections has $r^{-1}$-bounded geometry (see Figure \ref{fig_of_X0}).
\end{example}

\subsection{Weighted H\"{o}lder spaces}

On a submanifold $X$ with $r^{-1}$-bounded geometry, we wish to characterize the asymptotic behavior of a $C^{k,\beta}$-function near the origin relative to a weight function of the form $r^{\tau}\rho^{-\tau+\delta}$. This leads to
the following definition.

\begin{definition}\label{def: doubly weighted}
    For $X\cap \rho^{-1}(0,A^{-1})$ with $r^{-1}$-bounded geometry, we define the $C^{k,\beta}_{\delta,\tau}$-norm of $\nu$ on $X\cap \rho^{-1}(0,A^{-1})$ by
    \begin{equation*}
        \|\nu\|_{C^{k,\beta}_{\delta,\tau}} = \sup_{R,S>0,\Omega_{R,S}\subset\rho^{-1}(0,A^{-1})} R^{-\tau}S^{\tau-\delta} \|\nu\|_{C^{k,\beta}_{R^{-2}g_X}(\Omega_{R,S})}.
    \end{equation*}
    Here the subscript $R^{-2}g_{X}$ indicates that we measure the H\"{o}lder norm using the uniformly bounded rescaled metric.
\end{definition}

As noted in Remark \ref{rmk: Lag nbhd for r^-1 bdd}, for a Lagrangian submanifold with $r^{-1}$-bounded geometry, we can define exact Lagrangian graphs for potential functions with suitable asymptotic control. The next proposition characterizes such functions via weighted H\"{o}lder spaces.

\begin{proposition}\label{prop: 2,2 can take graph}
    Let $X$ be a Lagrangian submanifold in $\C^{n+1}$ with $r^{-1}$-bounded $C^K$-geometry. There exists $c_0>0$ such that if $\|\nu\|_{C^{K+2,\beta}_{2,2}}=c<c_0$, then the Lagrangian graph $Y=\gr_X(d\nu)$ is well-defined and itself has $r^{-1}$-bounded $C^K$-geometry. Moreover, for every $k\leq K$,
    \[
        |\nabla^k_{g_Y}(\cF^*g_X-g_Y)| < \Psi(c)r^{-k}
    \]
    where $\cF: X\rightarrow Y$ is the nearest-point projection and $\Psi(c)\rightarrow 0$ as $c\rightarrow 0$.
\end{proposition}

\begin{proof}
    Since $X$ has $r^{-1}$-bounded $C^K$-geometry, the rescalings $(R^{-1}\Omega_{R,S},R^{-2}g_{X})$ have uniformly bounded $C^K$-geometry. Hence there exists $c_0>0$ such that for any function $\tilde \nu$ on $R^{-1}\Omega_{R,S}$ with
    \begin{equation}\label{eqn: tilde nu bdd}
        \|\tilde \nu\|_{C_{R^{-2}g_{X}}^{K+2,\beta}} < c_0,
    \end{equation}
    the Lagrangian graph of $d\tilde \nu$ is well-defined and has bounded geometry.

    Now, if $\nu$ satisfies $\|\nu\|_{C_{2,2}^{K+2,\beta}} = c < c_0$, then for every $\Omega_{R,S}$ we have
    \begin{equation}\label{eqn: R^2nu bdd}
        R^{-2}\|\nu\|_{C^{K+2,\beta}_{R^{-2}g_X}} \leq c < c_0.
    \end{equation}
    Under the scaling $x \mapsto R^{-1}x$, the $1$-form scales as $d_{g_{R^{-1}X}}\tilde \nu = R^{-1}d_{g_X}\nu$ so the potential function scales as $\tilde \nu = R^{-2}\nu$ (up to an additive constant which is set to $0$ by $|\nu| = o(R^2)\rightarrow 0$ as $R\rightarrow 0$). Hence, \eqref{eqn: R^2nu bdd} implies \eqref{eqn: tilde nu bdd}, the graph is well‑defined and has bounded geometry; scaling back shows that $Y$ is well-defined and has $r^{-1}$-bounded $C^K$-geometry.
    
    (2). The estimate for the difference of metrics follows from the fact that, in the rescaled picture which has bounded geometry, the $k$-th derivative of $\cF^*g_X - g_Y$ is controlled by the $(k+2)$-th derivatives of $\tilde\nu$, which are of order $c$. Scaling back introduces a factor $R^{-k}$, and because $r \sim R$ on $\Omega_{R,S}$, we obtain the desired bound with $\Psi(c) \to 0$ as $c \to 0$.
\end{proof}

\begin{corollary}\label{cor: close geom & preserve TC}
    Suppose $X\cap \rho^{-1}(0,A_0^{-1})$ has $r^{-1}$-bounded $C^K$-geometry and that there exist constants $C,\kappa>0$, independent of $A\geq A_0$, such that in $X\cap \rho^{-1}(0,A^{-1})$,
    \begin{equation}\label{eqn: 2,2 to delta,tau}
        r^{\tau} \rho^{\delta-\tau} \leq CA^{-\kappa}r^2.
    \end{equation}
    Let $\nu\in C_{\delta,\tau}^{K+2,\beta}\bigl(X\cap \rho^{-1}(0,A^{-1})\bigr)$ and set $Y=\gr_X (d\nu)$. Then
    \begin{enumerate}
        \item Given $\epsilon > 0$, for $A$ sufficiently large we have
        \[
            |\nabla^k_{g_Y}(\cF^*g_X-g_Y)| < \epsilon r^{-k},\quad k\leq K,
        \]
        where $\cF: X\rightarrow Y$ is the nearest-point projection.
        \item Tangent cone(s) is preserved: for any sequence $\rho_i\searrow 0$,
        \[
            \lim_{\rho_i\searrow 0} \rho_i^{-1}(Y) = \lim_{\rho_i\searrow 0} \rho_i^{-1}X.
        \]
    \end{enumerate}
\end{corollary}

\begin{proof}
    By the definition of weighted spaces, condition \eqref{eqn: 2,2 to delta,tau}  implies the following comparison of weighted norms:
    \begin{equation}\label{eqn: 2,2 bdd by delta,tau}
        \|\nu\|_{C_{2,2}^{K+2,\beta}} \leq CA^{-\kappa}\|\nu\|_{C_{\delta,\tau}^{K+2,\beta}}.
    \end{equation}
    Taking $A$ large enough makes $\|\nu\|_{C_{2,2}^{K+2,\beta}}$ arbitrarily small so part (1) directly follows from Proposition \ref{prop: 2,2 can take graph}. For part (2), the comparison of weighted norms \eqref{eqn: 2,2 bdd by delta,tau} implies
    \[
        \rho^{-1}|d\nu| \leq \|d\nu\|_{C_{1,1}^{1,\beta}} \leq \|\nu\|_{C_{2,2}^{2,\beta}} \leq CA^{-\kappa}\|\nu\|_{C_{\delta,\tau}^{2,\beta}} \leq CA^{-\kappa}\rightarrow 0\quad \text{as }A\to \infty,
    \]
    which guarantees the tangent cone(s) of $\gr_X(d\nu)$ coincide with those of $X$.
\end{proof}

Next we note that an $n$-th order differential operator determined by the local geometry of $X$ defines bounded maps from $C^{k+n,\beta}_{\delta,\tau}$ to $C^{k,\beta}_{\delta-n,\tau-n}$. In particular, the Lagrangian angle operator enjoys the following properties:

\begin{proposition}\label{prop: bdd op & Lip est}
    Suppose $X\cap \rho^{-1}(0,A_0^{-1})$ has $r^{-1}$-bounded $C^K$-geometry and take $\delta,\tau$ as in corollary \ref{cor: close geom & preserve TC}. Then for $A\geq A_0$ sufficiently large:
    \begin{enumerate}
        \item The Lagrangian angle operator and the Laplacian on $X$ define bounded linear operators
        \[
            LA_{X},\Delta_{X}: C^{k+2,\beta}_{\delta,\tau}(X\cap \rho^{-1}(0,A^{-1})) \rightarrow C^{k,\beta}_{\delta-2,\tau-2}(X\cap \rho^{-1}(0,A^{-1})).
        \]
        \item Moreover, $LA_X$ can be written as
        \begin{equation}\label{eqn: form of LA_X}
            LA_{X}(\nu)[x] = \theta_X + \Delta_{X}\nu(x) + Q_{X}\bigl(x,\nu,\nabla \nu\bigr),
        \end{equation}
        where the nonlinear part $Q_{X}$ satisfies the Lipschitz estimate,
        \[
            \|Q_{X}(\nu_1) - Q_{X}(\nu_2)\|_{C^{k,\beta}_{\delta-2,\tau-2}} \leq C(\|\nu_1\|_{C^{k+2,\beta}_{2,2}} + \|\nu_2\|_{C^{k+2,\beta}_{2,2}})\|\nu_1-\nu_2\|_{C^{k+2,\beta}_{\delta,\tau}}
        \]
        for all $\nu_1,\nu_2\in C^{k+2,\beta}_{\delta,\tau}(X\cap \rho^{-1}(0,A^{-1}))$.
    \end{enumerate}
\end{proposition}

\begin{proof}
    (1). Fix a function $\nu\in C^{k+2,\beta}_{\delta,\tau}$. By Corollary \ref{cor: close geom & preserve TC}, the graph of $d\nu$ is well-defined for $A$ sufficiently large. On all $\Omega_{R,S}$,
    \begin{equation}\label{eqn: delta, tau norm leq 1}
        R^{-\tau}S^{\tau-\delta} 
        \|\nu\|_{C^{k+2,\beta}_{g_{R^{-1}X}}(\Omega_{R,S})} \leq \|\nu\|_{C^{k+2,\beta}_{\delta,\tau}}.
    \end{equation}
    We know that $g_{R^{-1}X}$ is uniformly bounded on all $R^{-1}\Omega_{R,S}$, so there exists $C>0$ independent of $R,S$ such that
    \begin{equation}\label{eqn: LA bdd on rescaling}
        R^{-\tau}S^{\tau-\delta} \|LA_{R^{-1}X} \nu\|_{C^{k,\beta}_{g_{R^{-1}X}}(\Omega_{R,S})} \leq CR^{-\tau}S^{\tau-\delta} \|\nu\|_{C^{k+2,\beta}_{g_{R^{-1}X}}(\Omega_{R,S})}.
    \end{equation}
    The Lagrangian angle operator is a second order differential operator, so under the scaling $x \mapsto R^{-1}x$, it scales as $R^{2}LA_X = LA_{R^{-1}X}$; hence
    \begin{equation}\label{eqn: scaling of LA}
        R^{2-\tau}S^{\tau-\delta} \|LA_{X} \nu\|_{C^{k,\beta}_{R^{-2}g_X}(\Omega_{R,S})} \leq R^{-\tau}S^{\tau-\delta} \|LA_{R^{-1}X} \nu\|_{C^{k,\beta}_{g_{R^{-1}X}}(\Omega_{R,S})}
    \end{equation}
    Following \eqref{eqn: scaling of LA}, \eqref{eqn: LA bdd on rescaling}, \eqref{eqn: delta, tau norm leq 1}, and taking supremum over all $R,S$, we have
    \[
        \|LA_{X} \nu\|_{C^{k,\beta}_{\delta-2,\tau-2}} \leq C
        \|\nu\|_{C^{k+2,\beta}_{\delta,\tau}},
    \]
    and the same argument applies to $\Delta_X$. 
    
    (2). The expansion \eqref{eqn: form of LA_X} follows from McLean’s theorem (Theorem \ref{thm: McL}) and the fact that $X$ has bounded geometry after rescaling. For the Lipschitz estimate, we again work on the rescaled pieces $R^{-1}\Omega_{R,S}$ with  $\tilde \nu = R^{-2}\nu$. Here,
    \[
        LA_{R^{-1}X}(\tilde\nu) = \theta_{R^{-1}X} + \Delta_{R^{-1}X}(\tilde \nu) + Q_{R^{-1}X}(\tilde\nu,\nabla_{R^{-1}X} (\tilde\nu)),
    \]
    where for $\|\tilde\nu\|_{C^{k+2,\beta}_{g_{R^{-1}X}}}$ sufficiently small, $Q_{R^{-1}X}$ is a convergent power series with at least degree-two terms. Indeed, $\|\tilde\nu_i\|_{C^{k+2,\beta}_{g_{R^{-1}X}}}=R^{-2}\|\nu_i\|_{C^{k+2,\beta}_{g_{R^{-1}X}}}$ is small by \eqref{eqn: 2,2 bdd by delta,tau}; hence $Q_{R^{-1}X}(\tilde \nu_1) - Q_{R^{-1}X}(\tilde \nu_2)$ is a convergent power series in which every term has a factor of $(\tilde\nu_1-\tilde\nu_2)$ or its derivative. It follows that
    \begin{equation*}
    \begin{split}
        \|Q_{R^{-1}X}(\tilde\nu_1) - Q_{R^{-1}X}(\tilde \nu_2)\|_{C^{k,\beta}_{g_{R^{-1}X}}}\leq C\|\tilde \nu_1-\tilde \nu_2\|_{C^{k+2,\beta}_{g_{R^{-1}X}}}\bigl(\|\tilde \nu_1\|_{C^{k+2,\beta}_{g_{R^{-1}X}}} + \|\tilde \nu_2\|_{C^{k+2,\beta}_{g_{R^{-1}X}}}\bigr).
    \end{split}
    \end{equation*}
    Scaling back and using the relation $Q_X(\nu) = Q_{R^{-1}X}(\tilde \nu)$ yields the desired estimate after taking weighted suprema.
\end{proof}

\subsection{The perturbation}

We can now prove the main perturbation proposition.

\begin{proposition}\label{prop: main pert}
    Suppose there exists a Lagrangian submanifold $X$ in $\C^{n+1}$ and $A_0\gg 1$ satisfying the following conditions:
    \begin{enumerate}
        \item $X\cap \rho^{-1}(0,A_0^{-1})$ is smooth away from the origin and has unique tangent cone $C\times\R$ at the origin.
        \item $X\cap \rho^{-1}(0,A_0^{-1})$ has $r^{-1}$-bounded $C^K$-geometry.
        \item There exist constants $\delta,\tau\in\R$, $C>0$, $\kappa > 0$ such that for every $A\geq A_0$ sufficiently large,
        \begin{enumerate}
            \item[(i)] In $X\cap \rho^{-1}(0,A^{-1})$, 
            \[
                \|\theta_{X}\|_{C^{0,\beta}_{0,0}} \leq CA^{-\kappa}r^{\tau-2} \rho^{\delta-\tau}\quad \text{and}\quad r^{\tau} \rho^{\delta-\tau}\leq CA^{-\kappa}r^2.
            \]
            \item[(ii)] The bounded linear operator
            \[
                \Delta_X: C_{\delta,\tau}^{2,\beta}(X\cap \rho^{-1}(0,A^{-1})) \rightarrow C_{\delta-2,\tau-2}^{0,\beta}(X\cap \rho^{-1}(0,A^{-1}))
            \]
            has a right inverse $P$, bounded independently of $A$.
        \end{enumerate}
    \end{enumerate}
    Then for $A$ sufficiently large, there exists a perturbation $Y$ of $X$ by an exact $1$-form with potential in $C_{\delta,\tau}^{2,\beta}$ such that $Y\cap \rho^{-1}(0,A^{-1})$ is a special Lagrangian submanifold, smooth away from the origin, with unique tangent cone $C\times\R$ at the origin.
\end{proposition}

\begin{remark}
    Combining two inequalities in (3)(i) implies that $\|\theta_{X}\|_{C^{0,\beta}_{0,0}}\leq CA^{-\kappa}$ for some $\kappa>0$. Recall that the H\"{o}lder norm is computed using $R^{-2}g_X$, so this means $\theta_X$ decays at the origin under a blowup to bounded geometry. For those familiar with mean curvature flow, this is analogous to saying the ``Type II blowup" of $X$ at the singularity is special Lagrangian, which is stronger than the fact that the tangent cone is special Lagrangian.
\end{remark}

\begin{remark}
    In (3)(i), condition $r^{\tau} \rho^{\delta-\tau}\leq CA^{-\kappa}r^2$ on $\delta,\tau$ ensures a perturbation in this weight is small enough such that it is well-defined in a neighborhood of the origin and preserves the tangent cone(s) (Corollary \ref{cor: close geom & preserve TC}). Condition $\|\theta_{X}\|_{C^{0,\beta}_{0,0}}\leq CA^{-\kappa}r^{\tau-2} \rho^{\delta-\tau}$ together with (3)(ii) means a perturbation in this weight is significant enough to kill the Lagrangian angle.
\end{remark}

\begin{proof}
    For $A$ sufficiently large, define the ball
    \[
        \cB_A = \{\nu\in C^{2,\beta}_{\delta,\tau}(X\cap \rho^{-1}(0,A^{-1})): \|\nu\|_{C^{2,\beta}_{\delta,\tau}} \leq A^{-\kappa/2}\}.
    \]
    We use a contraction mapping argument to show the Lagrangian angle equation
    \[
        LA_{X}(\nu) = \theta_X + \Delta_{X}(\nu) + Q_{X}(\nu) = 0
    \]
    has a solution in $\cB_A$. By condition (3)(ii), $\Delta_X$ has a right inverse $P$ bounded independently of $A$. Define the map
    \[
        \cN(\nu) = -P(\theta_X + Q_X(\nu)).
    \]
    If $\nu$ is a fixed point of $\cN$, then
    \begin{equation*}
    \begin{split}
        LA_X(\nu) &= \theta_X + \Delta_X(\nu) + Q_X(\nu) = \theta_X + Q_X(\nu) + \Delta_X(\cN(\nu))\\
        &= \theta_X + Q_X(\nu) + \Delta_X(-P(\theta_X + Q_X(\nu))) = 0
    \end{split}
    \end{equation*}
    i.e. $\nu$ solves $LA_X(\nu)=0$. We now show that for $A$ large enough, $\cN$ is a contraction mapping of $\cB_A$ into itself. First, using the boundedness of $P$ and the assumption on $\theta_X$ from condition (3)(i), we have for $A$ sufficiently large,
    \begin{equation}\label{eqn: stab of 0 under cN}
         \|\cN(0)\|_{C^{2,\beta}_{\delta,\tau}} = \|P(\theta_X)\|_{C^{2,\beta}_{\delta,\tau}} = C\|\theta_X\|_{C^{0,\beta}_{\delta-2,\tau-2}} \leq CA^{-\kappa} \leq \frac{1}{4}A^{-\kappa/2}.
    \end{equation}
    Using the Lipschitz estimate (Proposition \ref{prop: bdd op & Lip est} (2)) and weight comparison \eqref{eqn: 2,2 bdd by delta,tau}, we have for $A$ sufficiently large,
    \begin{equation}\label{eqn: cN contract}
    \begin{split}
         \|\cN(\nu_1)-\cN(\nu_2)\|_{C^{2,\beta}_{\delta,\tau}} &\leq C\|Q(\nu_1)-Q(\nu_2)\|_{C^{2,\beta}_{\delta-2,\tau-2}}\\
         &\leq C(\|\nu_1\|_{C^{2,\beta}_{2,2}} + \|\nu_2\|_{C^{2,\beta}_{2,2}})\|\nu_1-\nu_2\|_{C^{2,\beta}_{\delta,\tau}}\\
         &\leq CA^{-\kappa}(\|\nu_1\|_{C^{2,\beta}_{\delta,\tau}} + \|\nu_2\|_{C^{2,\beta}_{\delta,\tau}})\|\nu_1-\nu_2\|_{C^{2,\beta}_{\delta,\tau}}\\
         &\leq CA^{-3\kappa/2} \|\nu_1-\nu_2\|_{C^{2,\beta}_{\delta,\tau}}\\
         &\leq \frac{1}{2}\|\nu_1-\nu_2\|_{C^{2,\beta}_{\delta,\tau}}.
    \end{split}
    \end{equation}
    \eqref{eqn: cN contract} says $\cN$ is a contraction, and together with \eqref{eqn: stab of 0 under cN} implies $\cN$ maps $\cB_A$ into itself. By the contraction mapping theorem, $\cN$ has a fixed point $\nu_0\in \cB_A$. Then $Y = \gr_X(d\nu_0)$ is an exact Lagrangian submanifold in $\rho^{-1}(0,A^{-1})$ satisfying $LA_X(\nu_0)=0$; hence, $Y$ is special Lagrangian. Smoothness of $Y$ away from the origin follows from construction, and Corollary \ref{cor: close geom & preserve TC} guarantees that $Y$ has unique tangent cone $C\times\R$ at the origin.
\end{proof}

\section{Constructing the approximate solution}\label{sec: approx sol}

In this section, we prove the main theorem \ref{thm: main}. It suffices to construct an approximate solution $X$ in a neighborhood of the origin that satisfies all the conditions in Proposition \ref{prop: main pert}. We dedicate the rest of the section to this construction.

\subsection{General picture}\label{subsec: gen pic}
Pick an integer $a > 1$. Intuitively, the success of the perturbation scheme means the first-order geometry of the special Lagrangian $Y$ near the origin should already be characterized by $X$. Thus, outside a small region near the singularity axis $\R_u$ of $C\times\R$, the first-order behavior of $X$ is determined by a homogeneous harmonic function of degree $>2$ over the tangent cone $C\times\R$.

\begin{itemize}
    \item[\textbf{Step 1:}] Define $X_1$ as the Lagrangian graph of $d\varphi_a$ over $C\times \R$, where $\varphi_a$ is the harmonic polynomial defined in \eqref{eqn: def varphi_a} (see Section \ref{subsec: harmonic graph}).
\end{itemize}

On the other hand, $C\times\R$ is singular along $\R_u$, while $X$ is smooth away from the origin, so the description of $X$ as a graph over $C\times\R$ necessarily breaks down in a small region near $\R_u$, where $X$ needs to be `smoothed out'. Naively, we would like to smooth out in each $u$-cross section by a scaling $u^aL$ of $L$, i.e., the submanifold $X_0$ defined in Example \ref{exm: X_0}. However, $X_0$ is not Lagrangian in $\C^{n+1}$, so we need to `correct' it to a Lagrangian submanifold.

\begin{itemize}
    \item[\textbf{Step 2:}] Define $X_2$ that coincides with $X_0$ in the first $n$-coordinates and the real part of the last coordinate. The imaginary part of the last coordinate is corrected such that $X_2$ is Lagrangian in $\C^{n+1}$ (see Section \ref{subsec: lag corr}).
\end{itemize}

Finally, we finish the construction of $X$ by gluing $X_1$ and $X_2$ in the region $\{|u|^b\leq r \leq 2|u|^b\}$ for some $b\in (1,a)$. $b>1$ corresponds to the fact that $\{r \geq 2|u|^b\}$, the region of the harmonic graph description, dominates $X$ when approaching the origin. $b<a$ corresponds to the fact that the interpolating region separates from the smoothing region when approaching the origin. This guarantees $X_2$ is graphical over $C\times\R\cap\{r \leq 2|u|^b\}$, allowing interpolation of $X_1$ and $X_2$ as graphs over $C\times\R$. We also crucially rely on the fact that the dominant terms of $\varphi_a$ and $G$ agree in $\{r\sim |u^b|\}$. Intuitively, $\varphi_a$ can be viewed as a natural `harmonization' of the leading term of $G$ (by same-order terms), and $a$ is required to be an integer for this harmonization to be a polynomial and smooth everywhere.

\begin{itemize}
    \item[\textbf{Step 3:}] Define $X$ as $X_1$ in $\{r \geq 2|u|^b\}$, as $X_2$ in $\{r \leq |u|^b\}$, and as the interpolation of $X_1$ and $X_2$ as graphs over $C\times\R$ in $\{|u|^b\leq r \leq 2|u|^b\}$.
\end{itemize}

Notice that we only define $X$ in a neighborhood of the origin $\rho^{-1}(0,A^{-1})$ where $A$ is a sufficiently large constant.

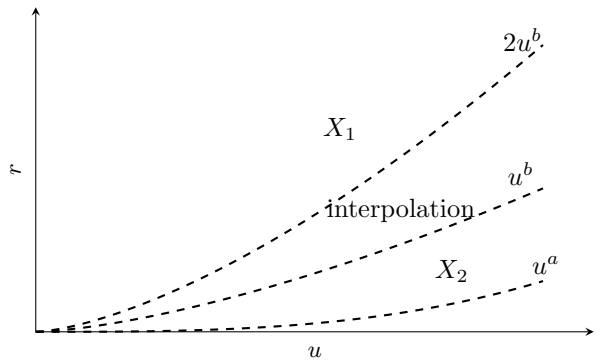
\begin{figure}[h]
\centering
\begin{tikzpicture}
\begin{axis}[
    xlabel={$u$},
    ylabel={$r$},
    xmin=0, xmax=0.55,
    ymin=0, ymax=0.8,
    axis lines=left,          % only left and bottom axes
    grid=none,                % no grid
    xtick=\empty,             % remove tick marks and numbers on x‑axis
    ytick=\empty,             % remove tick marks and numbers on y‑axis
    width=9cm, height=5.9cm,    % make the plot area a square
    clip=false,               % allow labels near edges if needed
    tick label style={font=\small},
    label style={font=\small},
]
% Plot u^a (solid line)
\addplot[domain=0:0.5, samples=100, smooth, thick, dashed] {x^3};

% Plot u^b (dashed line)
\addplot[domain=0:0.5, samples=100, smooth, thick, dashed] {x^(1.5)};

% Plot 2u^b (dotted line)
\addplot[domain=0:0.5, samples=100, smooth, thick, dashed] {2*x^(1.5)};

% Labels for the curves (placed near the right end)
\node at (axis cs:0.48,0.12) [anchor=south west] {$u^a$};
\node at (axis cs:0.5,0.34) [anchor=south east] {$u^b$};
\node at (axis cs:0.48,0.67) [anchor=south]      {$2u^b$};

% Region labels
\node at (axis cs:0.41,0.15) {$X_2$};                    % between u^a and u^b
\node at (axis cs:0.36,0.3) {interpolation};            % between u^b and 2u^b
\node at (axis cs:0.3,0.5)  {$X_1$};                    % above 2u^b

\end{axis}
\end{tikzpicture}
\caption{$X$ in different regions}
\label{X in diff regions}
\end{figure}

\subsection{Scaling analysis}

We first introduce a variation of the classical $\Psi$-notation.

\begin{definition}
    Let $\Psi^p(\epsilon)$ denote a function decaying \textbf{polynomially} as $\epsilon\rightarrow 0$.
    \[
        \Psi^p(\epsilon) \leq C\epsilon^{\kappa}
    \]
    for some $C,\kappa > 0$ where $C,\kappa$ are independent of $\epsilon$ for sufficiently small $\epsilon$. The only relevant application of this notation is for $\epsilon = A^{-1}$,
    \[
        \Psi^p(A^{-1}) \leq CA^{-\kappa}
    \]
    for some $C,\kappa > 0$ where $C,\kappa$ are independent of $A$ for sufficiently large $A$.
\end{definition}

In the following sections, we will perform the following procedures.

\begin{itemize}
    \item Given a base special Lagrangian submanifold $X$ with $r^{-1}$-bounded geometry, we construct (or realize) $Y$ as a graph $Y=\gr_X(dF)$ for some $F:X\rightarrow \R$ in some region $\cU$. We derive estimates on $F$ and its derivatives to show the graph is well-defined.
    \item Using the derivative estimates on $F$ together with McLean's theorem, we estimate the Lagrangian angle of $Y$ in $\cU$.
    \item From the derivative estimates on $F$, we deduce that the geometry of $Y$ is close to that of $X$; consequently, $Y$ also has $r^{-1}$-bounded geometry in $\cU$.
\end{itemize}

Since all submanifolds involved possess $r^{-1}$-bounded geometry (see Definition \ref{def: r^-1 bd geom}) and we wish to estimate the Lagrangian angle in weighted H\"{o}lder spaces (which use the rescaled metric $R^{-2}g_X$), we will carry out the procedures above locally in the sets $\Omega_{R,S}$ under the rescaling by $R^{-1}$, which we denote by $R^{-1}\Omega_{R,S}$. This approach is natural in view of the geometry we are working with. When performing estimates in $R^{-1}(\Omega_{R,S}\cap \rho^{-1}(0,A^{-1}))$, all the constants that appear, $C,C_k,\kappa>0$, will be taken independent of $R,S$ for $A$ sufficiently large.

We now introduce notations for the rescaling by $R^{-1}$. Write
\[
    \tilde X = R^{-1}X
\]
and introduce rescaled coordinates
\[
    \tilde z = R^{-1}z,\quad \tilde r =|\tilde z| = R^{-1}|z| = R^{-1}r,\quad \tilde w = R^{-1}w.
\]
The corresponding rescaled metric is
\[
    R^{-2}g_X(z,w) = g_{\tilde X}(R^{-1}z,R^{-1}w) = g_{\tilde X}(\tilde z, \tilde w),
\]
and the covariant derivatives are denoted by
\[
    \nabla = \nabla_{g_X}, \quad \tilde \nabla = \nabla_{g_{\tilde X}}.
\]
If $Y=\gr_X(dF)$ is the graph of $dF$ over $X$ for a function $F:X\rightarrow \R$, we define the rescaled potential $\tilde F:\tilde X\rightarrow \R$ (locally in $R^{-1}\Omega_{R,S}$) by
\begin{equation}\label{eqn: def rescaled pot}
     \tilde F(\tilde z,\tilde w) = R^{-2}F(z,w),
\end{equation}
which implies
\begin{equation}\label{eqn: rescaled derivative}
    \tilde \nabla^k \tilde F(\tilde z,\tilde w) = R^{-2}\tilde \nabla^k F(z,w) = R^{-2}R^k \nabla^k F(R\tilde z,R\tilde w) = R^{k-2}\nabla^k F(z,w).
\end{equation}
Taking $k=1$, we get the desired rescaling of $Y$:
\[
    \tilde Y = \gr_{\tilde X}(d\tilde F) = R^{-1}\gr_X(dF) = R^{-1}Y.
\]

\begin{remark}\label{rmk: small tilde F close geom}
    The definition of $\tilde F$ can also be understood as follows:
    \begin{equation*}
    \begin{split}
        \|F\|_{C^{k,\beta}_{2,2}(X\cap\rho^{-1}(0,A^{-1}))} &= \sup_{R,S>0,\Omega_{R,S}\subset\rho^{-1}(0,A^{-1})} R^{-2} \|F\|_{C^{k,\beta}_{R^{-2}g_X}(\Omega_{R,S})}\\
        &= \sup_{R,S>0,\Omega_{R,S}\subset\rho^{-1}(0,A^{-1})} \|\tilde F\|_{C^{k,\beta}_{g_{\tilde X}}(\Omega_{R,S})}.
    \end{split}
    \end{equation*}
    Hence, by Proposition \ref{prop: 2,2 can take graph}, to show $Y=\gr_X(dF)$ is well-defined and has $C^K$-geometry arbitrarily close to $X$ in $\rho^{-1}(0,A^{-1})$, it suffices to show the first $K+2$ derivatives of $\tilde F$ are arbitrarily small for $A$ sufficiently large.
\end{remark}

\subsection{Harmonic graph $X_1$}\label{subsec: harmonic graph}

Recall that $a>1$ is an integer. Define a polynomial $\varphi_a$ on $C$ such that on each connected component $C_i = \cone(\Sigma_i)$ of $C$,
\begin{equation}\label{eqn: def varphi_a}
    \varphi_{a} = -\left(u^{2a} + \sum_{j=1}^{a} a_{j}u^{2a-2j}r^{2j}\right)c_{\infty,m},
\end{equation}
where $a_j$ are coefficients uniquely determined by the condition $\Delta_{C\times\R}\varphi_{a} = 0$.

\begin{lemma}\label{lem: est varphi_a}
    In all $R^{-1}\Omega_{R,S}$,
    \[
        |\tilde \nabla^k \tilde \varphi_{a}| \leq C_k S^{2a}R^{-2}(S^{-1}R)^k.
    \]
\end{lemma}

\begin{proof}
    After rescaling by $R^{-1}$,
    \[
        \tilde\nabla^k  \tilde \varphi_{a} = R^{k-2}\nabla^k \varphi_{a}.
    \]
    \textbf{Case 1:} In the subregion $\{r\geq |u|\}$, $R\sim S$. The dominant term is the last term, i.e. the term with highest degree in $r$, so
    \[
        |\tilde\nabla^k  \tilde \varphi_{a}| \leq C_k R^{k-2} R^{2a-k} \leq C_k S^{2a}R^{-2}(S^{-1}R)^k.
    \]
    \textbf{Case 2:} In the subregion $\{r\leq |u|\}$, $|u|\sim S$. The dominant term is the first term, i.e. the term with highest degree in $u$, so
    \[
        |\tilde\nabla^k  \tilde \varphi_{a}| \leq C_k R^{k-2} S^{2a-k} \leq C_k S^{2a}R^{-2}(S^{-1}R)^k.
    \]
\end{proof}

\begin{remark}\label{rmk: varphi arb small}
    In $R^{-1}\bigl(\Omega_{R,S}\cap \{r \geq |u|\}\cap \rho^{-1}(0,A^{-1})\bigr)$, $R\sim S$,
    \[
        |\tilde \nabla^k \tilde \varphi_{a}| \leq C_k S^{2a-2} \leq C_kA^{-(2a-2)}.
    \]
    Pick $\Lambda\geq \Lambda_0$ as in Definition \ref{def: AC_submfld}. In $R^{-1}\bigl(\Omega_{R,S}\cap \{\Lambda|u|^a\leq r\leq |u|\}\cap \rho^{-1}(0,A^{-1})\bigr)$, $S\sim |u| \leq \Lambda^{-1}r \sim \Lambda^{-1}R$,
    \[
        |\tilde \nabla^k \tilde \varphi_{a}| \leq C_k |u|^{2a}R^{-2} \leq C_k\Lambda^{-2}.
    \]
    By choosing $A,\Lambda$ sufficiently large, we can make the derivatives sufficiently small such that the graph of $d\varphi_a$ is well-defined in $\{r \geq \Lambda|u|^a\}\cap \rho^{-1}(0,A^{-1})$ as
    \[
        X_1 = \gr_{C\times\R}(d\varphi_{a}).
    \]
\end{remark}

\begin{corollary}\label{cor: Lag ang of X_1}
    In all $R^{-1}\bigl(\Omega_{R,S}\cap \{r\geq |u|^b\}\cap \rho^{-1}(0,A^{-1})\bigr)$,
    \[
        \|\theta_{X_1}\|_{C^1_{R^{-2}g_{X_1}}} \leq S^{2a}R^{-2}\Psi^p(A^{-1}).
    \]
\end{corollary}

\begin{proof}
    Note that $\widetilde{C\times \R}$ is special Lagrangian and $\tilde \varphi_{a}$ is harmonic. Hence McLean's theorem (Theorem \ref{thm: McL}) implies the Lagrangian angle of $\tilde X_1 = \gr_{C\times\R}(d\tilde \varphi_a)$, and only contain the quadratic and higher order terms:
    \[
        \theta_{\tilde X_1} = \theta_{\gr_{\widetilde{C\times\R}}(d\tilde \varphi_a)} = Q(\tilde \nabla \tilde \varphi_{a}, \tilde \nabla^2 \tilde \varphi_{a}).
    \]
    By choosing $A$ sufficiently large such that $\Lambda|u|^a\leq  |u|^b$ in $\rho^{-1}(0,A^{-1})$, we can apply Lemma \ref{lem: est varphi_a} which implies all derivative of $\tilde \varphi_{a}$ are bounded by $C_kS^{2a}R^{-2}$. Thus
    \[
        \|\theta_{X_1}\|_{C^1_{R^{-2}g_{X_1}}} = \|\theta_{\tilde X_1}\|_{C^1_{g_{\tilde X_1}}} \leq C(S^{2a}R^{-2})^2 = S^{2a}R^{-2}\Psi^p(A^{-1}).
    \]
\end{proof}

\subsection{Lagrangian correction $X_2$ of $X_0$}\label{subsec: lag corr}

Recall that $L$ is an exact special Lagrangian submanifold asymptotic to cone $C$ with rate $\mu< 0$, with Liouville potential $\beta_L:L\rightarrow\R$ satisfying $d\beta_L = \lambda|_{L}$. We use $\phi: L\rightarrow \C^n$ to denote its isometric embedding into $\C^n$.

\begin{definition}\label{def: Phi_h}
    Let $U\subset\R$ be an open set and $h:U\rightarrow(0,\infty)$ be a smooth function. Define $\Phi_h:L\times U\rightarrow \C^n\times\C$ by
    % \[
    %     \Phi(x,u) = \left(|u|^a\phi(x),u - i2au^{2a-1}\beta_L(x)\right).
    % \]
    \[
        \Phi_h(x,u) = \left(h(u)\phi(x),u + iv_h(x,u)\right) \quad \text{where} \quad v_h(x,u) = -2h(u)'h(u)\beta_L(x).
    \]
    Since $\phi$ is a smooth embedding, $\Phi_h$ is also a smooth embedding.
\end{definition}

\begin{lemma}\label{lem: Phi is Lag}
    $\Phi_h$ is a smooth Lagrangian embedding in $\C^{n+1}$.
\end{lemma}

\begin{proof}
    Let $x_1,....,x_n$ be local coordinates on $L$. A basis $\{v_1,...,v_{n+1}\}$ of $T_p(\operatorname{Image}(\Phi_h))\cong \C^{n}\times\C$, $p=\Phi_h(x,u)$ is defined by
    \begin{equation*}
    \begin{split}
        v_k &= \partial_{x_k}\Phi_h = (h(u) \partial_k\phi(x), i\partial_k v(x,u)) \in \C^n\times\C, \quad\quad  k=1,...,n\\
        v_{n+1} &= \partial_u \Phi_h = (h'(u)\phi(x), 1+iv(x,u)) \in \C^n\times\C
    \end{split}
    \end{equation*}
    First, for all $k,l = 1,...,n$,
    \[
        \omega_{\C^{n+1}}(v_k,v_\ell) = \omega_{\C^n}(h(u) \partial_k\phi, h(u) \partial_\ell\phi) = h(u)^2\omega_{\C^n}(\partial_k\phi, \partial_\ell\phi) = 0
    \]
    where the third equality follows from fact that $\phi$ is a Lagrangian embedding. Next, for all $k = 1,...,n$, 
    \begin{equation*}
    \begin{split}
        \omega_{\C^{n+1}}(v_{n+1},v_k) &= \omega_{\C^n}(h' \phi, h \partial_k\phi) + \partial_kv_h = 2h'h \omega_{\C^n}\left(\frac{1}{2}\phi^{\perp}, \partial_k\phi\right) + \partial_k v_h
    \end{split}
    \end{equation*}
    where the definition of the Liouville structure implies
    \begin{equation}\label{eqn: liouv str}
        \omega_{\C^n}\left(\frac{1}{2}\phi^{\perp},\partial_k\phi\right) = \iota_{v_\lambda}\omega_{\C^n}|_L\left(\partial_k\phi\right) = \lambda|_L(\partial_{k}\phi) = d\beta_L(\partial_{k}) = \partial_k\beta_L.
    \end{equation}
    Then $v_h = -2h'h\beta_L$ implies $\omega_{\C^{n+1}}(v_{n+1},v_k) = 2h'h\partial_k\beta_L - 2h'h\partial_k\beta_L = 0$, hence $\Phi_h$ is an Lagrangian embedding.
\end{proof}

\begin{definition}[$X_2$ as an embedding]\label{def: X_2 embed}
    Consider the smooth function $h:\R\setminus\{0\}\rightarrow(0,\infty)$, $h(u) = |u|^a$. Define $\Phi = \Phi_h$ with $v=v_h=-2au^{2a-1}\beta_L$ and
    \[
        X_2 = \operatorname{Image} \Phi \cup \{(z,0)\in\C^n\times\C, z\in C\setminus\{0\}\}.
    \]
\end{definition}

\begin{remark}\label{rmk: prop of X_2}
    We make the following remarks concerning the definition of $X_2$.
    \begin{enumerate}
        \item (Symmetry) $h(u)=|u|^a$ is an even function, so $X_2$ is symmetric with respect to $\{u = 0\}$. Thus when proving properties of $X_2$, it suffices to only consider the case $u>0$ where $u=|u|$. When $u>0$, we have
        \[
            \Phi(x,u) = \left(u^a\phi(x),u - i2au^{2a-1}\beta_L(x)\right).
        \]
        \item (Regularity) The smoothness of $\phi$ and $h$ in $\{u\neq 0\}$ implies $X_2\cap \{u\neq 0\}$ is smooth. One can show $C^{2,\alpha}$-regularity of $X_2$ at $\{u = 0\}$ (using the equivalent Definition \ref{def: X_2 graph} of $X_2$). However, for our purpose of constructing the approximate solution $X$, which is defined as $X_1$ in an open neighborhood of $\{u=0,r\neq 0\}$ (see Figure \ref{X in diff regions}), the regularity of $X_2$ in $\{u\neq 0\}$ is sufficient.
        \item (Lagrangian correction) The first $n$ coordinates of $X_2$ are just scalings $u^aL$ of $L$. This coincide with $X_0$ defined in Example \ref{exm: X_0}. $X_2$ differ from $X_0$ only in the imaginary part of the last coordinate $v(x,u) = -2au^{2a-1}\beta_L(x)$, which is designed to make $X_2$ Lagrangian in $\C^{n+1}$ as above. This is why $X_2$ is called the Lagrangian correction of $X_0$. More generally, as in the proof of \ref{lem: Phi is Lag}, the Lagrangian condition enforces
        \[
            \partial_kv_h(x,u) = -2h'(u)h(u)\partial_k\beta_L(x)
        \]
        which determines $v_h$ up to a function only depending on $u$. In other words $v_h$ is uniquely determined by the Lagrangian condition up to a choice of normalization of $\beta_L$ for each $u$-cross-section:
        \begin{equation}\label{eqn: Lag corr}
             v_h(x,u) = -2h(u)'h(u)(\beta_L(x) + c(u)).
        \end{equation}
        \item (Hamiltonian isotopy) Consider $u$ as a time parameter, then $\Psi_u: L\rightarrow\C^n$ defined by $\Psi_u(x) = h(u)\phi(x)$ is a time-dependent diffeomorphism with normal velocity
        \[
            V_h(\tilde x) = h'(u)\phi(x)^\perp = h'(u)h(u)^{-1}(h(u)\phi(x))^\perp
        \]
        Since $L$ is embedded, we can define a Hamiltonian function $H:\C^n\times\R\rightarrow\R$ whose restriction at time $u$ to the scaling $h(u)L$ is $-v(x,u)$:
        \[
            H(\tilde x,u) = H_u(\tilde x) = 2h'(u)h(u)\beta_L(x) = 2h'(u)h(u)^{-1}\beta_{h(u)L}(\tilde x)
        \]
        for all $\tilde x = h(u)x\in h(u)L$. Then \eqref{eqn: liouv str} implies the normal velocity $V_H$ of the Hamiltonian isotopy coincide with $V_h$.
        \item (Distance from $\R_u$) Recall that $L$ is chosen such that $\dist(0,L)\geq 1$. Consequently, $X_2$ is disjoint from the singular axis $\R_u$ with distance
        \begin{equation}\label{r>u^a}
            r|_{X_2} \geq |u|^a
        \end{equation}
    \end{enumerate}
\end{remark}

Next, we estimate the Lagrangian angle of $X_2$ in $\{r\leq 4u^b\}$.

\begin{lemma}\label{lem: Lag ang of X_2}
    In all $R^{-1}(\Omega_{R,S}\cap\{r\leq 4u^b\})$,
    \[
        \|\theta_{X_2}\|_{C^1_{R^{-2}g_{X_2}}} \leq S^{2a}R^{-2}\Psi^p(A^{-1}).
    \]
\end{lemma}

\begin{proof}
    We may write the basis $\{v_k\}_{k=1}^{n+1}$ as complex vectors in $\C^{n+1}$.
    \begin{equation*}
    \begin{split}
        v_k &= \left(u^a \partial_k\phi_1, ..., u^a \partial_k\phi_n, -i2au^{2a-1}\partial_k\beta_L\right), \quad\quad  k=1,...,n\\
        v_{n+1} &= \left(au^{a-1}\phi_1, ...,au^{a-1}\phi_n, 1-i2a(2a-1)u^{2a-2}\beta_L\right).
    \end{split}
    \end{equation*}
    The Lagrangian angle can be explicitly computed as the argument of the complex determinant of the matrix consisting of this basis.
    \begin{equation*}
    \begin{split}
        &\theta_{X_2} (\Phi(x,u)) = \arg \det \left( v_1^T \cdots v_n^T  v_{n+1}^T \right)\\
        &= \arg\det\begin{pmatrix}
            u^a \partial_1 \phi_1 & \cdots & u^a \partial_n \phi_1 & a u^{a-1} \phi_1 \\
            \vdots & & \vdots & \vdots \\
            u^a \partial_1 \phi_n & \cdots & u^a \partial_n \phi_n & a u^{a-1} \phi_n \\
            -i2a u^{2a-1} \partial_1{\beta_L} & \cdots & -i2a u^{2a-1} \partial_n{\beta_L} & 1-i2a(2a-1) u^{2a-2} \beta_L
            \end{pmatrix}\\
        &= \arg\det\begin{pmatrix}
            \partial_1 \phi_1 & \cdots & \partial_n \phi_1 & a u^{-1} \phi_1 \\
            \vdots & & \vdots & \vdots \\
            \partial_1 \phi_n & \cdots & \partial_n \phi_n & a u^{-1} \phi_n \\
            -i2a u^{2a-1} \partial_1{\beta_L} & \cdots & -i2a u^{2a-1} \partial_n{\beta_L} & 1 - i2a(2a-1) u^{2a-2} \beta_L
            \end{pmatrix}\\
        &= \arg\left( \det(\nabla\phi^T) + u^{2a-2} \det\left(\begin{array}{c|c}
            \nabla \phi^T & a\phi^T \\ \hline
            -i2a\nabla\beta_L & -i2a(2a-1)\beta_L
            \end{array}\right) \right)\\
        &= \arg \left( \det(\nabla\phi^T) + u^{2a-2}\det(\nabla\phi^T)\left( - i2a(2a-1)\beta_L + i2a^2\nabla\beta_L \cdot (\nabla\phi^T)^{-1} \cdot \phi^T \right)\right)\\
        &= \arg (\det(\nabla\phi^T)) + \arg \left( 1 + u^{2a-2}\left( - i2a(2a-1)\beta_L + i2a^2\nabla\beta_L \cdot (\nabla\phi^T)^{-1} \cdot \phi^T \right)\right)\\
        &= \arg\left( 1 + u^{2a-2} E(x) \right) = \arctan\left(u^{2a-2}E(x)\right)
    \end{split}
    \end{equation*}
    The third to last equality uses the determinant formula for bock matrices. The last equality uses the assumption that $\phi$ is special Lagrangian, i.e. $\arg (\det(\nabla\phi^T)) = 0$.
    
    Next, we analyze
    \[
        E(x) = - i2a(2a-1)\beta_L + i2a^2\nabla\beta_L \cdot (\nabla\phi^T)^{-1} \cdot \phi^T.
    \]
    By the asymptotic condition of $\beta_L$ (Corollary \ref{cor: asymp of beta_L}) and the fact that $\phi$ is isometric,
    \[
        |\beta_L| \leq C,\quad  |\nabla\beta_L| \leq C(1+r)^{\mu-1},\quad |\nabla^2\beta_L|\leq C(1+r)^{\mu-2}
    \]
    \[
        |\phi(x)| = r, \quad |\nabla \phi| = 1, \quad |\nabla^2 \phi| = r^{-1}
    \]
    where $\mu< 0$. Then we can bound $|E(x)|$ and $|\nabla E(x)|$,
    \begin{equation*}
    \begin{split}
        |E(x)| &\lesssim |\beta_L| + |\nabla\beta_L| |\nabla \phi|^{-1}|\phi| \leq C + C(1+r)^{\mu-1} r \leq C\\
        |\nabla E(x)| &\lesssim |\nabla \beta_L| + |\nabla^2\beta_L| |\nabla \phi|^{-1} |\phi| + |\nabla \beta_L| |\nabla^2 \phi||\nabla \phi|^{-2} |\phi| + |\nabla \beta_L| |\nabla \phi|^{-1} |\nabla\phi|\\
        &\leq C(1+r)^{\mu-1} + C(1+r)^{\mu-2}\cdot r + C(1+r)^{\mu-1}\cdot r^{-1}\cdot r + C(1+r)^{\mu-1}\\
        &\leq C(1+r)^{\mu-1}.
    \end{split}
    \end{equation*}
    Consequently, in $\Omega_{R,S}\cap\{r\leq 4u^b\}$, where $u\sim S$, $R\sim r\lesssim u^b\ll u$,
    \begin{equation*}
        |\theta_{X_2}(\Phi(x,u))| = |\arctan (u^{2a-2}E(x))| \lesssim u^{2a-2}|E(x)| = S^{2a}R^{-2}\Psi^p(A^{-1})
    \end{equation*}
    \begin{equation*}
    \begin{split}
        |\nabla_{R^{-2}g_{X_2}}\theta_{X_2}(\Phi(x,u))| &\leq CR|\partial_u\theta_{X_2}(\Phi(x,u))| + CRu^{-a}|\nabla_x\theta_{X_2}(\Phi(x,u))|\\
        &\leq CRu^{2a-3}|E(x)| + CRu^{-a}u^{2a-2}|\nabla E(x)|\\
        &\leq Cu^{2a}R^{-2}(u^{-1}R)^3 + Cu^{2a}R^{-2} u^{-2-a}R^3 (1+(u^{-a}R))^{\mu-1}\\
        &\leq Cu^{2a}R^{-2}(u^{-1}R)^3 + Cu^{2a}R^{-2} u^{-2}R^2\\
        &= S^{2a}R^{-2}\Psi^p(A^{-1}).
    \end{split}
    \end{equation*}
    Hence, we have the desired estimate.
\end{proof}

Next, using Definition \ref{def: AC_submfld} , we introduce the following equivalent definition of $X_2$ as a small graph over the model space $C\times\R$. This definition more straightforwardly exhibit the asymptotic geometry of $X_2$.

\begin{definition}[$X_2$ as graphs over model space]\label{def: X_2 graph}
    For $A,\Lambda > 0$ sufficiently large, in $\rho^{-1}(0,A^{-1})\cap \{r > \Lambda |u|^a\}$, $X_2$ is a graph over $C\times\R$, i.e.
    \[
        X_2 = \gr_{C\times\R}(dG)
    \]
    where $G: C\times\R\rightarrow \R$ is defined by
    \[
        G(x,u) = \begin{cases}
            -u^{2a}c_{\infty,m} + u^{2a}g(|u|^{-a}x)& u\neq 0\\
            0& u = 0
        \end{cases}
    \]
\end{definition}

Here, $g$ is the potential function of $L$ from Definition \ref{def: AC_submfld} such that $L\cap \{r\geq \Lambda\} = \gr_{C\cap \{r\geq \Lambda\}}(dg)$. A priori, the graph is only well-defined when $dG$ is sufficiently small, which we will show below. Note that while $g$ decay at rate $\mu<0$ when $|x|=r\rightarrow\infty$, $G$ only converges to constants when $r\rightarrow\infty$.

\begin{lemma}\label{lem: est G}
    In all $R^{-1}(\Omega_{R,S}\cap \rho^{-1}(0,A^{-1})\cap \{r>\Lambda |u|^a\}$,
    \[
        |\tilde \nabla^k \tilde G| \leq C_k u^{2a}R^{-2}(u^{-1}R)^k + C_ku^{2a}R^{-2}(u^{-a}R)^\mu.
    \]
\end{lemma}

\begin{proof}
    After rescaling by $R^{-1}$,
    \[
        \tilde G = R^{-2}G = -R^{-2}u^{2a}c_{\infty,m} + R^{-2}u^{2a}g(|u|^{-a}x)
    \]
    First, the derivative of the first term bounded as in the proof of Lemma \ref{lem: est varphi_a},
    \begin{equation}\label{eqn: est u^2a der}
        |\tilde\nabla^k  (R^{-2}u^{2a})| \leq C_k u^{2a}R^{-2}(u^{-1}R)^k,
    \end{equation}
    so it suffices to only consider the second term in $\tilde G$. By Definition \ref{def: AC_submfld},
    \begin{equation}\label{eqn: est g der}
    \begin{split}
        |\tilde \nabla^k g(|u|^{-a}x))| &\leq C_kR^k (u^{-a}R)^{\mu-k}(\nabla_{\tilde x}(|u|^{-a}r)^k + \partial_{\tilde u}(|u|^{-a}r)^k)\\
        &\leq C_k u^{ak} (u^{-a}R)^{\mu}(u^{-ak}+R^ju^{-ak-k})\\
        &\leq C_k(u^{-a}R)^{\mu}(1+ (u^{-1}R)^k).
    \end{split}
    \end{equation}
    The derivative of the second term of $\tilde G$ can be written as
    \[
        \tilde\nabla^k  (R^{-2}u^{2a}g(|u|^{-a}x)) = \sum_{i+j=k}\tilde\nabla^i (R^{-2}u^{2a}) \tilde\nabla^j  g(|u|^{-a}x)).
    \]
    Then \eqref{eqn: est u^2a der} and \eqref{eqn: est g der} together implies the desired estimate.
\end{proof}

\begin{remark}\label{rmk: dG arb small}
    Plugging in $k=1$ and using the fact that $R > C\Lambda u^a$, we see that
    \[
        |\tilde \nabla \tilde G| \leq C_k \Lambda^{-2} + C_k u^{2a-1}R^{-1},
    \]
    which can be made arbitrarily small by taking $A$ and $\Lambda$ large enough, so the graph of $dG$ is well-defined. In particular, in $\{\Lambda |u|^a\leq r\leq 2|u|^b\}$, $u^{-1}R \ll 1$, so the same estimate holds for all $k\geq 1$,
    \begin{equation}\label{eqn: der est of tilde G}
        |\tilde \nabla^k \tilde G| \leq C_k \Lambda^{-2}.
    \end{equation}   
\end{remark}

\begin{lemma}
    Definition \ref{def: X_2 embed} and Definition \ref{def: X_2 graph} are equivalent in the corresponding subregion.
\end{lemma}

\begin{proof}
    First, observe that by definition of $g$, the first $n$-coordinates of $X_2 = \gr_{C\times\R}(dG)$ are also scalings $u^aL$. So the embedding definition \ref{def: X_2 embed} and the graphical definition \ref{def: X_2 graph} coincide in the first $n$-coordinates and the real part of the ($n+1$)-th coordinate.
    
    Second, observe that both definitions are Lagrangian in $\C^{n+1}$ by construction, so as noted in Remark \ref{rmk: prop of X_2}(3), the imaginary part of the ($n+1$)-th coordinate is uniquely defined up to a function only depending on $u$.

    Lastly, observe that when $x\rightarrow\infty_m$, the asymptotics of $g$ and $\beta_L$ implies
    \[
        \partial_u G(x,u) = -2au^{2a-1}c_{\infty,m} + \partial_u(u^{2a}g(x)) \rightarrow -2au^{2a-1}c_{\infty,m},
    \]
    \[
        v(x,u) = -2au^{2a-1}\beta_L(x) \rightarrow -2au^{2a-1}c_{\infty,m},
    \]
    so for both definitions, given any $u$, the imaginary part of the ($n+1$)-th coordinate converge to $-2au^{2a-1}c_{\infty,m}$ when $x\rightarrow\infty_m$. Combined with the second observation, the two definitions are equivalent.
\end{proof}

\subsection{Interpolation between $X_1$ and $X_2$}

From the previous two subsections:
\begin{enumerate}
    \item $X_1$ is the graph of $d\varphi_a$ over $C\times\R$ in $\{r\geq \Lambda|u|^a\}$,
    \item $X_2$ is the graph of $dG$ over $C\times\R$ in $\{r\geq \Lambda|u|^a\}$.
\end{enumerate}
Since $b < a$, we can choose $A$ large enough such that $\Lambda |u|^a < |u|^b$ in $\rho^{-1}(0,A^{-1})$. Then both $X_{1}$ and $X_2$ are graphs over $C\times\R$ in $\{|u|^b\leq r\leq 2|u|^b\}$, so we can interpolate between them here as graphs over $C\times\R$:
\begin{equation}\label{eqn: def of X}
    X = 
    \begin{cases}
        X_1 & r \geq 2|u|^b,\\
        \gr_{C\times\R}(I) & |u|^b\leq r \leq 2|u|^b,\\
        X_2 & r \leq |u|^b,
    \end{cases}
\end{equation}
where
\[
    I (z,u) = \chi_2\left(\frac{|z|}{|u|^b}\right) G(z,u) + \left(1- \chi_2\left(\frac{|z|}{|u|^b}\right)\right)\varphi_a(z,u)
\]
and $\chi_2:\R \rightarrow \R$ is a smooth cut-off function satisfying
\[
    \chi_2(t) =
    \begin{cases}
        0\quad t\geq 2,\\
        1\quad t\leq 1.
    \end{cases}
\]

We now estimate $I$ and its derivatives, from which we can show the graph is well-defined and obtain Lagrangian angle estimates. The crucial fact is that the dominant terms of $\varphi_a$ and $G$ are both $-u^{2a}c_{\infty,m}$.

\begin{lemma}\label{lem: est G-varphi}
    In all $R^{-1}(\Omega_{R,S}\cap \{|u|^b\leq r\leq 2|u|^b\}\cap \rho^{-1}(0,A^{-1}))$,
    \[
        \left|\tilde \nabla^k (\tilde G - \tilde \varphi_a)\right| \leq S^{2a}R^{-2}\Psi^p(A^{-1})
    \]
\end{lemma}

\begin{proof}
    Taking the difference of $G$ and $\varphi_a$, the leading term $-u^{2a}c_{\infty,m}$ cancels,
    \[
        G - \varphi_a = u^{2a}g(u^{-a}x) + u^{2a}\sum_{j=1}^{a} a_ju^{-2j}r^{2j}c_{\infty,m}.
    \]
    In $\Omega_{R,S}\cap \{|u|^b\leq r\leq 2|u|^b\}$, we have $R\sim r\sim |u|^b\ll |u|\sim S$, so $\sum_{j=1}^{a} a_ju^{-2j}r^{2j}$ and its derivatives are bounded by the term with highest degree in $u$. Hence, after rescaling by $R^{-1}$,
    \begin{equation*}
    \begin{split}
        |\tilde\nabla^k  (\tilde G - \tilde \varphi_{a})| &= R^{k-2}|\nabla^k (G - \varphi_{a})|\\
        &\leq R^{k-2}u^{2a}u^{-ak}|\nabla^k g(u^{-a}x)| + C_k R^{k-2}u^{2a}u^{-2}r^{2-k}\\
        &\leq C_ku^{2a}R^{-2}(u^{-a}R)^k(1+u^{-a}R)^{\mu-k} +  C_k u^{2a}R^{-2}(u^{-1}R)^2\\
        &\leq C_ku^{2a}R^{-2}(u^{(b-a)\mu} + u^{2(b-1)})\\
        &= S^{2a}R^{-2}\Psi^p(A^{-1})
    \end{split}
    \end{equation*}
\end{proof}

\begin{corollary}\label{cor: est I}
    In all $R^{-1}(\Omega_{R,S}\cap \{|u|^b\leq r\leq 2|u|^b\}\cap \rho^{-1}(0,A^{-1}))$, for $k\geq 1$,
    \[
        \left|\tilde \nabla^k \tilde I\right| \leq S^{2a}R^{-2}\Psi^p(A^{-1})
    \]
\end{corollary}

\begin{proof}
    After rescaling by $R^{-1}$,
    \[
        \tilde I = \tilde \varphi_a - \chi_2\left(R^{1-b}\frac{|\tilde z|}{|\tilde u|^b}\right) \left(\tilde G-\tilde \varphi_a\right).
    \]
    We estimate the derivatives of the two terms separately. First, by $r\sim R\sim |u|^b$,
    \[
        \left|\tilde\nabla^{k}\left(R^{1-b}\frac{|\tilde z|}{|\tilde u|^b}\right) \right | \leq C_k R^{1-b}|\tilde u|^{-b} \leq C_kRS^{-b} \leq C_k.
    \]
    for all $k\geq 0$. Since $\chi_2$ is fixed, this shows all derivatives of the cutoff factors are uniformly bounded. Then by Lemma \ref{lem: est G-varphi},
    \[
        \left| \tilde\nabla^k\left(\chi_2\left(R^{1-b}\frac{|\tilde z|}{|\tilde u|^b}\right) \left(\tilde G-\tilde \varphi_a\right)\right)\right| \leq S^{2a}R^{-2}\Psi^p(A^{-1})
    \]
    for all $k\geq 0$. Next, by Lemma \ref{lem: est varphi_a} and $R\sim S^b$,
    \[
         |\tilde \nabla^k \tilde \varphi_{a}| \leq C_k S^{2a}R^{-2}(S^{-1}R)^k \leq S^{2a}R^{-2}\Psi^p(A^{-1})
    \]
     for all $k\geq 1$. Combining the two estimates above completes the proof.
\end{proof}

\begin{remark}\label{rmk: I arb small}
    Immediately from Corollary \ref{cor: est I} and $S^a \ll u^b\sim R$, for all $k\geq 0$, $\left|\tilde \nabla^k \tilde I\right|$ can be made arbitrarily small by taking $A$ sufficiently large. Consequently, the graph of $dI$ is well-defined.
\end{remark}

\begin{corollary}\label{cor: Lag ang of interpolation}
    The Lagrangian angle of $X$ satisfies
    \[
        \|\theta_{X}\|_{C^1_{R^{-2}g_{X}}} \leq S^{2a}R^{-2}\Psi^p(A^{-1}).
    \]
\end{corollary}

\begin{proof}
    In $\{r\geq 2|u|^b\}$, $X=X_1$, the estimate is proved in Corollary \ref{cor: Lag ang of X_1}. In $\{r\leq |u|^b\}$, $X=X_2$, the estimate is proved in Lemma \ref{lem: Lag ang of X_2}. We are left with the region $\{|u|^b\leq r\leq 2|u|^b\}$, where $X=\gr_{C\times\R}(dI)$. By McLean's theorem \ref{thm: McL} and the fact that $\widetilde{C\times\R}$ is special Lagrangian, the Lagrangian angle and its derivatives are controlled by $\left|\tilde \nabla^k \tilde I\right|$, which satisfies the desired estimate by Corollary \ref{cor: est I}.
\end{proof}

The following propositions translate the Lagrangian angle estimate of Corollary \ref{cor: Lag ang of interpolation} into weighted Hölder bounds and verify that $X$ satisfies the weight condition required by the perturbation scheme.

\begin{proposition}\label{prop: weight est Lag ang}
    Suppose $\delta$ is sufficiently close to $2a$ and $\tau$ is sufficiently close to $0$. Then for $A$ sufficiently large, there exist constants $C,\kappa > 0$ independent of $A$ such that in $X\cap\rho^{-1}(0, A^{-1})$,
    \begin{equation*}
        \|\theta_{X}\|_{C^{0,\beta}_{0,0}} < CA^{-\kappa}\rho^{\delta - \tau}r^{\tau - 2}.
    \end{equation*}
\end{proposition}

\begin{proof}
    Set $\delta_0=2a$ and $\tau_0 = 0$. By Corollary \ref{cor: Lag ang of interpolation}, there exist $C,\kappa > 0$ such that
    \begin{equation*}
        \|\theta_{X}\|_{C^{0,\beta}_{0,0}} \leq \|\theta_{X}\|_{C^1_{R^{-2}g_{X}}} \leq \Psi^p(A^{-1})\rho^{2a}r^{-2} \leq CA^{-\kappa}\rho^{\delta_0 - \tau_0}r^{\tau_0 - 2}.
    \end{equation*}
    Because $\delta$, $\tau$ are chosen close to $\delta_0$, $\tau_0$, we can absorb the small difference into the factor $A^{-\kappa}$ (possibly reducing $\kappa$) to obtain the desired estimate.
\end{proof}

\begin{proposition}\label{prop: weight is graphical}
    Suppose $\delta > 2a$ and $\tau$ is sufficiently close to $0$ (depending on $\delta$). Then for $A$ sufficiently large, there exist constants $C,\kappa > 0$ independent of $A$ such that in $X\cap\rho^{-1}(0, A^{-1})$,
    \[
        r^{\tau-2} \rho^{\delta-\tau}\leq CA^{-\kappa}.
    \]
\end{proposition}

\begin{proof}
    Set $\tau_0 = 0$, then
    \[
        r^{\tau_0-2} \rho^{\delta-\tau_0}\leq C r^{-2}\rho^{\delta} \leq C \rho^{-2a+\delta} \leq CA^{-\kappa},
    \]
    where $\kappa = \delta-2a > 0$. In the second inequality, we used $r\gtrsim \rho^a$ on $X$, which is implied by \eqref{r>u^a}. For $\tau$ close enough to $\tau_0$, we can absorb the small difference into the factor $A^{-\kappa}$ (possibly reducing $\kappa$) to obtain the desired estimate.
\end{proof}

The following proposition compares the geometry of $X$ to that of two special Lagrangian model spaces. When far from the singular axis $\R_u$, $X$ is close $C\times\R$; when close to the singular axis $\R_u$ and $u\approx u_0$, $X$ is close $u_0^aL\times\R$. For simplicity, we denote $S_0 := u_0^aL\times\R$.

\begin{proposition}\label{prop: geometry of X}
    Given any $\epsilon > 0$, for $A$ sufficiently large, there exists $\Lambda>0$ such that the following holds. Define the sets
    \[\mathcal{U} =\{r > \Lambda |u|^a\},\quad \mathcal{V}_{u_0} =\{r < 2\Lambda |u|^a, ||u|-u_0|\leq u_0^a\}, \quad 0<u_0\ll 1.\]
    \begin{enumerate}
        \item In $\cU\cap\rho^{-1}(0,A^{-1})$, for all $k\leq K$,
        \[
            |\nabla^k(\cG^*g_{X} - g_{C\times \R})| \leq \epsilon r^{-k},
        \]
        where $\cG:C\times \R\rightarrow X$ is the nearest-point projection.
        \item In $\cV_{u_0}\cap\rho^{-1}(0,A^{-1})$, for all $k\leq K$,
        \[
            |\nabla^k( \cF^*g_{X} - g_{S_{0}})| \leq \epsilon r^{-k},
        \]
        where $\cF:S_{0}\rightarrow X$ is the nearest-point projection.
    \end{enumerate}
    Consequently, $X\cap\rho^{-1}(0,A^{-1})$ has $r^{-1}$-bounded $C^K$-geometry.
\end{proposition}

\begin{proof}
    (1) In $\cU\cap\rho^{-1}(0,A^{-1})$,
    \[
        X = \begin{cases} \gr_{C\times\R}(d\varphi_a) & \{r\geq 2|u|^b\}\\
            \gr_{C\times\R}(dI) & \{|u|^b\leq r\leq 2|u|^b\}\\
            \gr_{C\times\R}(dG) & \{r\leq |u|^b\}
        \end{cases}
    \]
    By \eqref{eqn: der est of tilde G}, Remark \ref{rmk: varphi arb small}, and Remark \ref{rmk: I arb small}, each of the graphs defining $X$ is arbitrarily $C^k$-close to $C\times\R$ in the respective region, provided $A$ is chosen large enough (after choosing $\Lambda$ large enough). As explained in Remark \ref{rmk: small tilde F close geom}, Proposition \ref{prop: 2,2 can take graph} then implies that the induced metric on $X$ as a graph is arbitrarily close to the metric of $C\times\R$ in the stated sense.
    
    (2). In $\mathcal{V}_{u_0}$, $X=X_2$ is defined by the embedding $\Phi$ in Definition \ref{def: X_2 embed}. We may also embed $S_0$ using $\Phi_h$ for $h\equiv u_0^a$ in Definition \ref{def: Phi_h}, we denote this definition by $\Phi_{S_0}:L\times\R\rightarrow \C^{n+1}$, i.e.
    \[
        \Phi_{S_0}(x,u) = \left(u_0^a\phi(x),u\right),
    \]
    \[
        \Phi(x,u) = \left(u^a\phi(x),u - i2au^{2a-1}\beta_L(x)\right).
    \]
    We would like to rescaled $\Phi$ and $\Phi_{S_0}$ by $R$ in a locality where $r\sim R$ and show the rescaling have $\epsilon$-close $C^k$-geometry for all $k$. Using the fact that $u_0^{-a}\lesssim r^{-1}$ in $\{r < 2\Lambda |u|^a\}$, it suffices to prove the proximity of rescalings by $u_0^{-a}$. Rescale both $\Phi$ and $\Phi_{S_0}$ by $u_0^{-a}$ to $\tilde\Phi$ and $\tilde\Phi_{S_0}$ gives
    \[
        \tilde\Phi_{S_0}(x,u) = \left(\phi(x),u_0^{-a}u\right)
    \]
    \[
        \tilde\Phi(x,u) = \left(u^au_0^{-a}\phi(x),u_0^{-a}u - i2au^{2a-1}u_0^{-a}\beta_L(x)\right).
    \]
    Taking their difference gives
    \[
        \tilde \Phi(x,u) - \tilde \Phi_{S_0}(x,u) = \left(\left(u^au_0^{-a}-1\right)\phi(x), - i2au^{2a-1}u_0^{-a}\beta_L(x)\right).
    \]
    Using the fact that $||u|-u_0|\leq u_0^a$ implies $|u^a u_0^{-a}-1| \leq Cu_0^{a-1}$ and $|u|\sim u_0$,
    \begin{equation}
    \label{eqn: t.Phi close to t.Phi_S0}
        \left|\tilde\Phi(x,u) - \tilde\Phi_{S_0}(x,u))\right| \leq Cu_0^{a-1} \leq \epsilon,
    \end{equation}
    where the last inequality follows from $u_0 < A^{-1}$ for $A$ sufficiently large. Under rescaling by $u_0^{-a}$, the rescaled derivative is
    \[
        \tilde\nabla = \nabla_{u_0^{-a}u^ax} + \partial_{u_0^{-a}u}
    \]
    Using $|u|\sim u_0$ again implies $\nabla_{u_0^{-a}u^ax}$ is equivalent to $\nabla_x$. Taking $\nabla_x$ derivatives, we get estimate of the same type because both $\phi$ and $\beta$ are uniformly defined. Moreover, taking $\partial_{u_0^{-a}u} = u_0^{a}\partial_u$ derivative only improve the estimate, so \eqref{eqn: t.Phi close to t.Phi_S0} holds for all derivatives: for $A$ sufficiently large,
    \[
        \left|\tilde\nabla_x^k \left(\tilde\Phi(x,u) - \tilde\Phi_{S_0}(x,u)\right)\right| \leq \epsilon.
    \]
    for all $k\leq K$. Rescaling back by $u_0^a$, and using the fact that $u_0^{-a}\lesssim r^{-1}$ in $\{r < 2\Lambda |u|^a\}$,
    \[
        |\nabla^k( \cF^*g_{X} - g_{S_{0}})| \leq \epsilon u_0^{-ak} \leq \epsilon r^{-k}.
    \]
\end{proof}

\begin{corollary}
\label{cor: X has unique TC}
    $X$ is smooth in a neighborhood of $0$, except at $0$, where it has unique tangent cone $C\times \R$.
\end{corollary}

\begin{proof}
    In the region $\{r\geq 2|u|^b\}$, $X=\gr_{C\times\R}(d\varphi_a)$. By Lemma \ref{lem: est varphi_a}, $|\tilde\nabla\tilde\varphi_a| \leq CS^{2a-b-1}\leq \Psi^p(A^{-1})$ is arbitrarily close to that of $C\times\R$ when $A$ is large, so the only tangent cone of $X\cap \{r\geq 2|u|^b\}$ is $(C\setminus\{0\})\times\R$. Moreover, $\varphi_a$ is a polynomial with non-negative integer order terms in $u$ and $r$, so $X$ is smooth. In $X\cap \{r\leq 2|u|^b\}$, any open set avoids $\{u=0\}$, so smoothness follows from smoothness of $\Phi$ and $I$ in $\{u\neq 0\}$.
    
    Since $b>1$, $X\cap \{r\geq 2|u|^b\}$ fills up $X$ when approaching the origin while $X\cap \{r\leq 2|u|^b\}$ collapse to the singular axis $\{0\}\times\R$, so the only tangent cone of $X$ is $(C\setminus\{0\})\times\R$ union the singular axis $\{0\}\times\R$, which is $C\times\R$.
\end{proof}

\subsection{Inverting the Laplacian}\label{subsec: invert Laplace}

The final ingredient required for the perturbation argument is the invertibility of the linearized operator $\Delta_X$. The argument is standard so we sketch the main steps. See Section 3.2.2-3.2.5 in \cite{szekelyhidi2021minimal} for more discussion and Section 5 and 6 in \cite{szekelyhidi2019degenerations} for full detail.

First, computing the Laplacian in polar coordinates and using the fact that eigenvalues of $\Delta_\Sigma$ are non-negative, one can conclude that there are no non-constant homogeneous harmonic functions on $C=\cone(\Sigma)$ with degrees in the interval $(2-n,0)$. The existence of this spectral gap implies invertibility of the Laplacian on the two model spaces.

\begin{proposition}
    For $\delta$ avoiding a discrete set of indicial roots, and $\tau\in(2-n,0)$, the Laplace operator
    \[
        \Delta_{C\times \R}: C^{k,\beta}_{\delta,\tau}(C\times \R) \rightarrow C^{k-2,\beta}_{\delta-2,\tau-2}(C\times \R)
    \]
    is invertible.
\end{proposition}

\begin{proposition}
    For $\delta$ avoiding a discrete set of indicial roots, and $\tau\in(2-n,0)$, the Laplace operator
    \[
        \Delta_{L\times \R}: C^{k,\beta}_{\delta,\tau}(L\times \R) \rightarrow C^{k-2,\beta}_{\delta-2,\tau-2}(L\times \R)
    \]
    is invertible.
\end{proposition}

Next, using the fact that $X$ is arbitrarily close to either $C\times\R$ or $S_{0}= u_0^aL\times\R$ in the respective regions (Proposition \ref{prop: geometry of X}), we can patch together the inverses on model spaces to construct an approximate inverse $P'$ for $\Delta_X$ in the sense that its composition with $\Delta_X$ minus the identity is a contraction. This step is similar to the procedure of constructing an approximate solution $X$. Lastly, using the approximate inverse $P'$, a contraction mapping argument shows the existence of an actual right inverse $P$. This step is similar to the proof of Proposition \ref{prop: main pert}.

\begin{proposition}\label{prop: invert Lap X}
    Let $\tau \in (2-n,0)$, and suppose that $\delta$ avoids a discrete set of indicial roots. Then for sufficiently large $A>0$, the Laplace operator
    \begin{equation*}
        \Delta_{X}: C^{2,\beta}_{\delta,\tau}(X\cap \{\rho\in(0,A^{-1})\}) \rightarrow C^{0,\beta}_{\delta-2,\tau-2}(X\cap \{\rho\in(0,A^{-1})\})
    \end{equation*}
    is surjective, with a right inverse $P$ bounded independently of $A$.
\end{proposition}

Finally, all the ingredients are complete, and we can prove the main theorem.

\begin{proof}[Proof of the main theorem \ref{thm: main}]
    We only need to verify that $X$ satisfies all conditions in Proposition \ref{prop: main pert}. First, Corollary \ref{cor: X has unique TC} guarantees that $X\cap \rho^{-1}(0,A^{-1})$ has unique tangent cone $C\times\R$ and is smooth away from the origin. Second, Proposition \ref{prop: geometry of X} shows that $X\cap \rho^{-1}(0,A^{-1})$ has $r^{-1}$-bounded $C^K$-geometry. Lastly, suppose
    \begin{itemize}
        \item $\delta>2a$ is sufficiently close to $2a$ and avoids a discrete set of indicial roots, and
        \item $\tau\in (2-n,0)$ is sufficiently close to $0$.
    \end{itemize}
    Then $\delta,\tau$ satisfies the conditions for Proposition \ref{prop: weight est Lag ang}, \ref{prop: weight is graphical}, and \ref{prop: invert Lap X}. Such $\delta,\tau$ clearly exist, so condition (3) on $X$ is satisfied. This finishes the proof.
\end{proof}

\begin{appendix}

\section{Disconnected examples: union of generic special Lagrangians with linearly intersecting tangent planes}\label{app: disconn exm}

Consider the degenerate case where $C=\cone(\Sigma) = P$ is a plane, as in \ref{subsec: harmonic graph}, we may still define a harmonic polynomial
\[
    \varphi_{a} = u^{2a} + \sum_{j=1}^{a} a_{j}u^{2a-2j}r^{2j}
\]
Since a plane as a cone is not singular at the vertex, the graph $X_1 = \gr_{P\times\R}\varphi_a$ is a globally defined smooth graph over the special Lagrangian plane $P\times\R$ in $\C^{n+1}$ with tangent plane $P\times\R$ at the origin. Lemma \ref{lem: est varphi_a} now applies to the globally-defined graph $X_1$, so the weighted estimate of the Lagrangian angle is also satisfied globally. Hence $X_1 = \gr_{P\times\R}\varphi_a$ fits into the conditions of Proposition \ref{prop: main pert}. Consequently, the perturbation scheme produces a smooth special Lagrangian graph $Y$ in a neighborhood of the origin with tangent plane $P\times\R$ at the origin.

If we start with multiple planes, $C = \bigcup_m P_m$, where $P_m$ are transverse intersecting special Lagrangian planes in $C^m$, then the above procedure produces a union $Y = \bigcup_m Y_m$ where $Y_m$ has tangent plane $P_m\times\R$ at the origin. Such $Y$ is a special Lagrangian submanifold in a neighborhood of the origin in $\C^{m+1}$ with unique cylindrical tangent cone $C\times\R = \bigcup_m P_m\times\R$, a union of special Lagrangian planes intersecting along a line. Unlike the special Lagrangian produced by the main theorem \ref{thm: main}, $Y\setminus\{0\}$ is not connected with disjoint components $Y_m\setminus \{0\}$.

\end{appendix}

\bibliographystyle{alpha}

\bibliography{final.bbl}

\end{document}